\documentclass{article}

\usepackage[numbers]{natbib}
\usepackage{authblk}
\usepackage{a4wide}

\usepackage[utf8]{inputenc} 
\usepackage[T1]{fontenc}    
\usepackage{hyperref}       
\usepackage{url}            
\usepackage{booktabs}       
\usepackage{amsfonts}       
\usepackage{nicefrac}       
\usepackage{microtype}      
\usepackage{breakcites}
\usepackage{extarrows}

\usepackage{amsmath} 
\usepackage{enumitem} 
\usepackage{dsfont} 
\usepackage{graphicx} 
\usepackage{subcaption}
\usepackage{amssymb} 
\usepackage{amsthm}
\usepackage{graphicx}
\usepackage{amssymb}
\usepackage{etoolbox}
\usepackage{color}
\usepackage{graphicx}
\usepackage{wrapfig}
\usepackage{blindtext}
\usepackage{stmaryrd}
\usepackage{lipsum}
\usepackage{algorithm2e}
\usepackage{caption}
\usepackage{comment}
\newcommand{\R}{\mathbb{R}}
\newcommand{\tr}{\mathrm{Tr}}
\newcommand{\N}{\mathbb{N}}
\newcommand{\E}{\mathbb{E}}
\newcommand{\h}{\mathcal{H}}

\renewcommand{\P}{\mathbb{P}}

\newcommand{\Poinca}{\mathcal{P}}
\newcommand{\var}{\mathop{ \rm var}}

\renewcommand{\(}{\left(}
\renewcommand{\)}{\right)}

\newtheorem{defi}{Definition}
\newtheorem{thm}{Theorem}
\newcommand{\bt}{\begin{thm}}
\newcommand{\et}{\end{thm}}

\newtheorem{prop}{Proposition}
\newtheorem{lemma}{Lemma}

\newtheorem{rmk}{Remark}

\newcommand{\BIT}{\begin{itemize}}
\newcommand{\EIT}{\end{itemize}}
\newcommand{\BNUM}{\begin{enumerate}}
\newcommand{\ENUM}{\end{enumerate}}
\newcommand{\BA}{\begin{array}}
\newcommand{\EA}{\end{array}}

\newcommand{\Diag}{\mathop{\rm Diag}}

\newtheorem{assshort}{\bfseries Ass.}
\newcommand{\bas}{\begin{assshort}}
\newcommand{\eas}{\end{assshort}}

\newcommand{\asm}[1]{{\textbf{(Ass.\ \ref{#1})}}}

\newcommand{\C}{\Sigma}
\newcommand{\D}{\Delta}

\newcommand{\Cn}{\widehat{\Sigma}}
\newcommand{\Dn}{\widehat{\Delta}}
\newcommand{\Nla}{\mathcal{N}_\infty(\lambda)}
\newcommand{\Fla}{\mathcal{F}_\infty(\lambda)}

\makeatletter
\newtheorem{ass}{\bfseries (A \hspace*{.5mm} ) $\!\!\!\!\!\!\!\!\!\!\!\!\!\!\!$ }
\patchcmd{\endass}{\@endpefalse}{}{}{}
\makeatother

\title{Statistical Estimation of the Poincaré constant and Application to Sampling Multimodal Distributions}

\author[1]{Loucas Pillaud-Vivien}
\author[1]{Francis Bach}
\author[2]{Tony Lelièvre}
\author[1]{Alessandro Rudi}
\author[2]{Gabriel Stoltz}
\affil[1]{INRIA - Ecole Normale Supérieure - PSL Research University}
\affil[2]{Université Paris-Est - CERMICS (ENPC) - INRIA}

\begin{document}

\maketitle

\begin{abstract}
Poincaré inequalities are ubiquitous in probability and analysis and have various applications in statistics (concentration of measure, rate of convergence of Markov chains). The Poincaré constant, for which the inequality is tight, is related to the typical convergence rate of diffusions to their equilibrium measure. In this paper, we show both theoretically and experimentally that, given sufficiently many samples of a measure, we can estimate its Poincaré constant. As a by-product of the estimation of the Poincaré constant, we derive an algorithm that captures a low dimensional representation of the data by finding directions which are difficult to sample. These directions are of crucial importance for sampling or in fields like molecular dynamics, where they are called reaction coordinates. Their knowledge can leverage, with a simple conditioning step, computational bottlenecks by using importance sampling techniques.
\end{abstract}

\section{Introduction}

Sampling is a cornerstone of probabilistic modelling, in particular in the Bayesian framework where statistical inference is rephrased as the estimation of the posterior distribution given the data~\citep{robert2007bayesian,murphy2012machine}: the representation of this distribution through samples is both flexible, as most interesting quantities can be computed from them (e.g., various moments or quantiles), and practical, as there are many sampling algorithms available depending on the various structural assumptions made on the model. Beyond one-dimensional distributions, a large class of these algorithms are iterative and update samples with a Markov chain which eventually converges to the desired distribution, such as Gibbs sampling or Metropolis-Hastings (or more general Markov chain Monte-Carlo algorithms~\citep{gamerman2006markov,gilks1995markov,durmus2017nonasymptotic}) which are adapted to most situations, or Langevin's~algorithm~\citep{durmus2017nonasymptotic,raginsky2017non,Welling2011BayesianLV,Mandt:2017:SGD:3122009.3208015,lelievre_stoltz_2016,bakry2014}, which is adapted to sampling from densities in $\R^d$.

While these sampling algorithms are provably converging in general settings when the number of iterations tends to infinity, obtaining good explicit convergence rates has been a central focus of study, and is often related to the mixing time of the underlying Markov chain~\citep{meyn2012markov}. In particular, for sampling from positive densities in $\R^d$, the Markov chain used in Langevin's algorithm can classically be related to a diffusion process, thus allowing links with other communities such as molecular dynamics~\citep{lelievre_stoltz_2016}. The main objective of molecular dynamics is to infer macroscopic properties of matter from atomistic models via averages with respect to probability measures dictated by the principles of statistical physics. Hence, it relies on high dimensional and highly multimodal probabilistic models.

When the density is log-concave, sampling can be done in polynomial time with respect to the dimension~\citep{ma2018sampling,durmus2017,durmus2017nonasymptotic}. However, in general, sampling with generic algorithms does not scale well with respect to the dimension. Furthermore, the multimodality of the objective measure can trap the iterates of the algorithm in some regions for long durations: this phenomenon is known as metastability. To accelerate the sampling procedure, a common technique in molecular dynamics is to resort to importance sampling strategies where the target probability measure is biased using the image law of the process for some low-dimensional function, known as ``reaction coordinate'' or ``collective variable''. Biasing by this low-dimensional probability measure can improve the convergence rate of the algorithms by several orders of magnitude \citep{Lelievre_2008,Lelievre2013}. Usually, in molecular dynamics, the choice of a good reaction coordinate is based on physical intuition on the model but this approach has limitations, particularly in the Bayesian context \citep{Chopin2012}. There have been efforts to numerically find these reaction coordinates \cite{gkeka2019ml}. Computations of spectral gaps by approximating directly the diffusion operator work well in low-dimensional settings but scale poorly with the dimension. One popular method is based on diffusion maps \citep{COIFMAN20065,Boaz2006,Rohrdanz2011}, for which reaction coordinates are built by approximating the entire infinite-dimensional diffusion operator and selecting its first eigenvectors.

In order to assess or find a reaction coordinate, it is necessary to understand the convergence rate of diffusion processes. We first introduce in Section \ref{sec:PI} Poincaré inequalities and Poincaré constants that control the convergence rate of diffusions to their equilibrium. We then derive in Section \ref{sec:Estimation} a kernel method to estimate it and optimize over it to find good low dimensional representation of the data for sampling in Section \ref{sec:learning_RC}. Finally we present in Section \ref{sec:experiments} synthetic examples for which our procedure is able to find good reaction coordinates.

\paragraph{Contributions.} In this paper, we make the following contributions:
\begin{itemize}
\item We show both theoretically and experimentally that, given sufficiently many samples of a measure, we can estimate its Poincaré constant and thus quantify the rate of convergence of Langevin dynamics. 
\item By finding projections whose marginal laws have the largest Poincaré constant, we derive an algorithm that captures a low dimensional representation of the data. This knowledge of ``difficult to sample directions'' can be then used to accelerate dynamics to their equilibrium measure.
\end{itemize}

\section{Poincaré Inequalities }
\label{sec:PI}

\subsection{Definition}

We introduce in this part the main object of this paper which is the Poincaré inequality \citep{bakry2014}. Let us consider a probability measure $d\mu$ on $\R^d$ which has a density with respect to the Lebesgue measure. Consider $H^1(\mu)$ the space of functions in $L^2(\mu)$ (i.e., which are square integrable) that also have all their first order derivatives in $L^2$, that is, $H^1(\mu) = \{f \in L^2(\mu),\ \int_{\R^d}f^2 d\mu+\int_{\R^d}\|\nabla f\|^2 d\mu < \infty \}$.

\begin{defi} [Poincaré inequality and Poincaré constant]
\label{defi:PI}
The Poincaré constant of the probability measure $d\mu$ is the smallest constant $\Poinca_{\mu}$ such that for all $f \in H^1(\mu) $ the following Poincaré inequality  \textbf{(PI)} holds:
\begin{align}
\label{eq:Poicare_constant}
\int_{\R^d}f(x)^2 d\mu(x) - &\left(\int_{\R^d}f(x) d\mu(x)\right)^2 \leqslant \Poinca_{\mu} \int_{\R^d}\|\nabla f(x)\|^2 d\mu(x).
\end{align}
\end{defi}

In Definition \ref{defi:PI} we took the largest possible and the most natural functional space $ H^1(\mu)$ for which all terms make sense, but Poincaré inequalities can be equivalently defined for subspaces of test functions $\h$ which are dense in $ H^1(\mu)$. This will be the case when we derive the estimator of the Poincaré constant in Section \ref{sec:Estimation}.

\begin{rmk}[A probabilistic formulation of the Poincaré inequality.]
Let $X$ be a random variable distributed according to the probability measure $d\mu$. \textbf{(PI)} can be reformulated as: for~all~$f \in H^1(\mu)$,
\begin{align}
\mathrm{Var}_\mu\,(f(X)) \leqslant \Poinca_{\mu}\,\E_\mu \left[\,\|\nabla f(X)\|^2\right].
\end{align}
Poincaré inequalities are hence a way to bound the variance from above by the so-called \textit{Dirichlet energy} $ \E \left[\,\|\nabla f(X)\|^2\right]$ (see \emph{\citep{bakry2014}}).
\end{rmk}

\subsection{Consequences of \textbf{(PI)}: convergence rate of diffusions}

Poincaré inequalities are ubiquitous in various domains such as probability, statistics or partial differential equations (PDEs). For example, in PDEs they play a crucial role for showing the existence of solutions of Poisson equations or Sobolev embeddings \citep{Gilbarg2001}, and they lead in statistics to concentration of measure results \citep{gozlan2010}. In this paper, the property that we are the most interested in is the convergence rate of diffusions to their stationary measure $d\mu$. In this section, we consider a very general class of measures: $d\mu (x) = \mathrm{e}^{-V(x)}dx$ (called Gibbs measures with potential $V$), which allows for a clearer explanation. Note that all measures admitting a positive density can be written like this and are typical in Bayesian machine learning \citep{robert2007bayesian} or molecular dynamics \citep{lelievre_stoltz_2016}. Yet, the formalism of this section can be extended  to more general cases \citep{bakry2014}.

Let us consider the overdamped Langevin diffusion in $\R^d$, that is the solution of the following stochastic differential equation (SDE):
\begin{align}
\label{eq:langevin}
\mathrm{d}X_t = -\nabla V (X_t) \mathrm{d}t + \sqrt{2}\,\mathrm{d} B_t,
\end{align}
where $(B_t)_{t\geqslant0}$ is a $d$-dimensional Brownian motion. It is well-known \citep{bakry2014} that the law of $(X_t)_{t\geqslant0}$ converges to the Gibbs measure $d\mu$ and that the Poincaré constant controls the rate of convergence to equilibrium in $L^2(\mu)$. Let us denote by $P_t (f) $ the Markovian semi-group associated with the Langevin diffusion $(X_t)_{t\geqslant0}$. It is defined in the following way: $P_t (f) (x) = \mathbb{E}[f(X_t)| X_0 = x]$. This semi-group satisfies the dynamics 
$$ \frac{d } {dt} P_t (f)= \mathcal{L} P_t (f), $$
where $\mathcal{L} \phi = \Delta^L \phi - \nabla V \cdot \nabla \phi$ is a differential operator called the infinitesimal generator of the Langevin diffusion \eqref{eq:langevin} ($\Delta^L$ denotes the standard Laplacian on $\R^d$). Note that by integration by parts, the semi-group $(P_t)_{t \geqslant 0}$ is reversible with respect to $d\mu$, that is: $-\int f(\mathcal{L}g)\,d\mu = \int \nabla f \cdot \nabla  g\, d\mu = -\int (\mathcal{L}f)g\,d\mu$. Let us now state a standard convergence theorem (see e.g.~\citep[Theorem 2.4.5]{bakry2014} ), which proves that $\Poinca_\mu$ is the characteristic time of the exponential convergence of the diffusion to equilibrium in~$L^2(\mu)$.
\begin{thm}[Poincaré and convergence to equilibrium]
\label{thm:Poinca_diffusions}
With the notation above, the following statements are equivalent:
\begin{enumerate}[label=$(\roman{*})$]
\item $\mu$ satisfies a Poincaré inequality with constant $\Poinca_\mu$;
\item For all $f$ smooth and compactly supported, $\mathrm{Var}_\mu (P_t (f)) \leqslant \mathrm{e}^{-2t / \Poinca_\mu} \mathrm{Var}_\mu (f)$ for all $t \geqslant 0$.
\end{enumerate}
\end{thm}

\begin{proof}
The proof is standard. Note that upon replacing $f$ by $f-\int\!f d\mu$, one can assume that $\int\!f d\mu = 0$. Then, for all $t \geqslant 0$,
\begin{align*}
\label{eq:variance_poinca}
\frac{d}{dt}\mathrm{Var}_\mu (P_t (f)) &= \frac{d}{dt}\int(P_t (f))^2 d\mu = 2 \int P_t (f) (\mathcal{L} P_t (f)) d\mu = -2 \int \|\nabla P_t (f)\|^2  d\mu \tag{$\ast$}
\end{align*}

Let us assume $(i)$. With equation \eqref{eq:variance_poinca}, we have $$ \frac{d}{dt}\mathrm{Var}_\mu (P_t (f)) = -2 \int \|\nabla P_t (f)\|^2  d\mu \leqslant -2\, \Poinca_\mu ^{-1}\int(P_t (f))^2 d\mu = -2\, \Poinca_\mu ^{-1} \mathrm{Var}_\mu (P_t (f)). $$ The proof is then completed by using Grönwall's inequality.

Let us assume $(ii)$. We write, for $t > 0$,
\begin{align*}
&-t^{-1}(\mathrm{Var}_\mu (P_t (f)) - \mathrm{Var}_\mu (f)) \geqslant -t^{-1}(\mathrm{e}^{-2t / \Poinca_\mu} - 1)\mathrm{Var}_\mu (f).
\end{align*}
By letting $t$ go to $0$ and using equation \eqref{eq:variance_poinca},
\begin{align*}
   2 \Poinca_\mu^{-1} \mathrm{Var}_\mu (f) &\leqslant \frac{d}{dt}\mathrm{Var}_\mu (P_t (f))_{t=0}  = 2 \int \|\nabla f \|^2 d\mu,
\end{align*}
which shows the converse implication.
\end{proof}

\begin{rmk}
Let $f$ be a centered eigenvector of $-\mathcal{L}$ with eigenvalue $\lambda \neq 0$. By the Poincaré inequality,
\begin{align*} 
\int f^2d\mu \leqslant \Poinca_\mu \int \|\nabla f\|^2 d\mu &= \Poinca_\mu \int  f(-\mathcal{L}f) d\mu = \Poinca_\mu \lambda \int  f^2 d\mu,
\end{align*}
from which we deduce that every non-zero eigenvalue of $-\mathcal{L}$ is larger that $1/\Poinca_\mu$.  The best Poincaré constant is thus the inverse of the smallest non zero eigenvalue of $-\mathcal{L}$. The finiteness of the Poincaré constant is therefore equivalent to a \emph{spectral gap} property of $-\mathcal{L}$. Similarly, a discrete space Markov chain with transition matrix $P$ converges at a rate determined by the spectral gap of $I-P$.  
\end{rmk}

There have been efforts in the past to estimate spectral gaps of Markov chains \citep{hsu2015mixing,levin2016estimating,qin2019estimating,wolfer2019estimating,combes2019computationally} but these have been done with samples from trajectories of the dynamics. The main difference here is that the estimation will only rely on samples from the stationary measure.

\paragraph{Poincaré constant and sampling.} In high dimensional settings (in Bayesian machine learning~\citep{robert2007bayesian}) or molecular dynamics \citep{lelievre_stoltz_2016} where~$d$ can be large -- from~$100$ to~$10^7$), one of the standard techniques to sample $d\mu(x) = \mathrm{e}^{-V(x)}dx$ is to build a Markov chain by discretizing in time the overdamped Langevin diffusion \eqref{eq:langevin} whose law converges to $d\mu$. According to Theorem~\ref{thm:Poinca_diffusions}, the typical time to wait to reach equilibrium is~$\Poinca_\mu$. Hence, the larger the Poincaré constant of a probability measure $d\mu$ is, the more difficult the sampling of $d\mu$ is. Note also that~$V$ need not be convex for the Markov chain to converge.

\subsection{Examples} 
\label{subsec:examples}

\paragraph{Gaussian distribution.} For the Gaussian measure on $\R^d$ of mean $0$ and variance $1$: $d\mu (x) = \frac{1}{(2\pi)^{d/2}} \mathrm{e}^{-\|x\|^2/2} dx$, it holds for all $f$ smooth and compactly supported,
\begin{align*}
\mathrm{Var}_{\mu}(f) \leqslant \int_{\mathbb{R}^d} \|\nabla f\|^2 d\mu, 
\end{align*}
and one can show that $\Poinca_\mu = 1$ is the optimal Poincaré constant (see \citep{chernoff1981}). More generally, for a Gaussian measure with covariance matrix $\Sigma$, the Poincaré constant is the spectral radius of $\Sigma$.

Other examples of analytically known Poincaré constant are $1/d$ for the uniform measure on the unit sphere in dimension $d$ \citep{Ledoux2014} and $4$ for the exponential measure on the real line \citep{bakry2014}. There also exist various criteria to ensure the existence of \textbf{(PI)}. We will not give an exhaustive list as our aim is rather to emphasize the link between sampling and optimization. Let us however finish this part with particularly important results.

\paragraph{A measure of non-convexity.} Let $d\mu(x) = \mathrm{e}^{-V(x)}dx$.  It has been shown in the past decades that the ``more convex'' $V$ is, the smaller the Poincaré constant is. Indeed, if $V$ is $\rho$-strongly convex, then the Bakry-Emery criterion \citep{bakry2014} tells us that $\Poinca_\mu \leqslant 1/\rho$. If $V$ is only convex, it has been shown that $d\mu$ satisfies also a \textbf{(PI)} (with a possibly very large Poincaré constant) \citep{Kannan1995,Bobkov1995}. Finally, the case where $V$ is non-convex is explored in detail in a one-dimensional setting and it is shown that for potentials $V$ with an energy barrier of height $h$ between two wells, the Poincaré constant explodes exponentially with respect the height $h$ \citep{menz2014}. In that spirit, the Poincaré constant of $d\mu(x) = \mathrm{e}^{-V(x)}dx$ can be a quantitative way to quantify how multimodal the distribution $d\mu$ is and hence how non-convex the potential $V$ is \citep{jain2017non,raginsky2017non}.

\section{Statistical Estimation of the Poincaré Constant}
\label{sec:Estimation}

The aim of this section is to provide an estimator of the Poincaré constant of a measure $\mu$ when we only have access to $n$ samples of it, and to study its convergence properties. More precisely, given $n$ independent and identically distributed (i.i.d.) samples $(x_1,\hdots,x_n)$ of the probability measure $d\mu$, our goal is to estimate $\Poinca_{\mu}$. We will denote this estimator (function of $(x_1,\hdots,x_n)$) by the standard notation $\widehat{\Poinca}_\mu$.

\subsection{Reformulation of the problem in a reproducing kernel Hilbert Space}

\paragraph{Definition and first properties.} Let us suppose here that the space of test functions of the \textbf{(PI)}, $\h$, is a reproducing kernel Hilbert space (RKHS) associated with a kernel $K$ on $\R^d$ \citep{smola-book,Cristianini2004}. This has two important consequences:
\begin{enumerate}
\item $\h$ is the linear function space $\h = \mathrm{span}\{ K(\cdot,x), \ x \in \R^d \}$, and in particular, for all $ x \in \R^d$, the function $y \mapsto K(y,x)$ is an element of $\h$ that we will denote by $K_x$.
\item The reproducing property: $\forall f \in \h$ and $\forall x \in \R^d$, $f(x) = \langle f,K(\cdot,x) \rangle_\h$. In other words, function evaluations are equal to dot products with canonical elements of the RKHS.
\end{enumerate}
We make the following mild assumptions on the RKHS:
\bas\label{asm:density} 
 \hspace*{.2cm}  The RKHS $\h$ is dense in $ H^1(\mu)$.
\eas
Note that this is the case for most of the usual kernels: Gaussian, exponential \citep{micchelli2006universal}. As \textbf{(PI)} involves derivatives of test functions, we will also need some regularity properties of the RKHS. Indeed, to represent $\nabla f$ in our RKHS we need a partial derivative reproducing property of the kernel space.

\bas\label{asm:regularity_RKHS} 
 \hspace*{.2cm}  $K$ is a Mercer kernel such that $K \in  C^2(\R^d \times \R^d)$.
\eas
Let us denote by $\partial_i = \partial_{x^i}$ the partial derivative operator with respect to the $i$-th component of $x$. It has been shown \citep{Zhou2008} that under assumption \asm{asm:regularity_RKHS}, $\forall i \in \llbracket1, d\rrbracket$, $\partial_i K_x \in \h$ and that a partial derivative reproducing property holds true: $\forall f \in \h$ and $\forall x \in \R^d$, $\partial_i f(x) = \langle \partial_i K_x  , f \rangle_\h $. Hence, thanks to assumption \asm{asm:regularity_RKHS}, $\nabla f$ is easily represented in the RKHS. We also need some boundedness properties of the kernel. 
\bas\label{asm:bounded_kernel} 
 \hspace*{.2cm}  $K$ is a kernel such that $\forall x \in~\R^d, \,K (x,x) \leqslant \mathcal{K}$ and\footnote{The subscript $d$ in $\mathcal{K}_d$ accounts for the fact that this quantity is expected to scale linearly with $d$ (as is the case for the Gaussian kernel).} $ \left\|\nabla K_x\right\|^2 \leqslant \mathcal{K}_d$, where $\left\|\nabla K_x \right\|^2: = \sum_{i=1}^d\langle \partial_i K_x, \partial_i K_x \rangle = \sum_{i=1}^d \frac{\partial^2 K}{\partial x^i \partial y^i} (x,x)$ (see calculations below), $x$ and $y$ standing respectively for the first and the second variables of $(x,y) \mapsto K(x,y)$.
\eas
The equality mentioned in the expression of $\|\nabla K_x\|^2$ arises from the following computation: $\partial_i K_y (x) = \langle \partial_i K_y, K_x \rangle = \partial_{y^i} K (x,y) $ and we can write that for all $x, y \in \R^d$, $\langle \partial_i K_x, \partial_i K_y \rangle = \partial_{x^i} \left( \partial_i K_y (x) \right) = \partial_{x^i} \partial_{y^i}  K(x,y)  $. Note that, for example, the Gaussian kernel satisfies \asm{asm:density}, \asm{asm:regularity_RKHS}, \asm{asm:bounded_kernel}.
\paragraph{A spectral point of view.} Let us define the following operators from $\h$ to $\h$: 
\begin{align*}
\C &= \E \left[K_x \otimes K_x \right],\hspace*{1.5cm} \D = \E \left[\nabla K_x \otimes_d \nabla K_x \right],
\end{align*}
and their empirical counterparts,
\begin{align*}
 \Cn  = \frac{1}{n} \sum_{i=1}^n K_{x_i} \otimes K_{x_i}, \hspace*{0.5cm} \Dn  = \frac{1}{n} \sum_{i=1}^n \nabla K_{x_i} \otimes_d \nabla K_{x_i},
\end{align*}
where $\otimes$ is the standard tensor product: $\forall f,g, h \in \h$, $(f \otimes g) (h) = \langle g, h \rangle_{_\h} f$ and $\otimes_d$ is defined as follows:  $\forall f,g \in \h^d$ and $h \in \h$, $(f \otimes_d g) (h) = \sum_{i=1}^d \langle g_i, h\rangle_{_\h} f_i $.

\begin{prop}[Spectral characterization of the Poincaré constant]
\label{prop:Spectral_characterization}
Suppose that assumptions \asm{asm:density}, \asm{asm:regularity_RKHS}, \asm{asm:bounded_kernel} hold true. Then the Poincaré constant $\Poinca_\mu$ is the maximum of the following Rayleigh ratio:
\begin{align}
\label{eq:spectral_poinca}
\Poinca_\mu = \sup_{f \in \h \setminus  \mathrm{Ker}(\Delta) } \frac{\langle f,C f \rangle_\mathcal{H}}{\langle f, \D f \rangle_\mathcal{H} } = \left\|\D^{-1/2} C \D^{-1/2}\right\|,
\end{align}
with $\|\cdot\|$ the operator norm on $\h$ and $C = \Sigma - m \otimes m$ where $ m = \int_{\R^d} K_x d\mu(x) \in \h $ is the covariance operator, considering $\Delta^{-1}$ as the inverse of $\Delta$ restricted to $\left(\mathrm{Ker} (\Delta) \right)^\perp$.
\end{prop}

Note that $C$ and $\D$ are symmetric positive semi-definite trace-class operators (see Appendix~\ref{subsec:operators}). Note also that $\mathrm{Ker} (\Delta)$ is the set of constant functions, which suggests introducing $\h_0 := (\mathrm{Ker} (\Delta))^\perp = \h \cap L^2_0(\mu)$, where $L^2_0(\mu)$ is the space of $L^2(\mu)$ functions with mean zero with respect to $\mu$. Finally note that $\mathrm{Ker} (\Delta) \subset \mathrm{Ker} (C)$ (see Section \ref{sec:proofs_of_prop_1/2} of the Appendix). With the characterization provided by Proposition \ref{prop:Spectral_characterization}, we can easily define an estimator of the Poincaré constant~$\widehat{\Poinca}_\mu$, following standard regularization techniques from kernel methods \citep{smola-book,Cristianini2004,fukumizu2007statistical}.

\begin{defi}
The estimator $\widehat{\Poinca}_\mu^{n,\lambda}$ of the Poincaré constant is the following:
\begin{align}
\label{eq:Empirical_rayleigh_ratio}
\widehat{\Poinca}_\mu^{n,\lambda} := \sup_{f \in \h \setminus \mathrm{Ker} (\Delta)} \frac{\langle f,\widehat{C} f \rangle_{\mathcal{H}}}{\langle f, (\widehat{\D} + \lambda I) f \rangle_{\mathcal{H}} } = \left\|\widehat{\D}_\lambda^{-1/2} \widehat{C} \widehat{\D}_\lambda^{-1/2}\right\|,
\end{align}
with $\widehat{C} = \Cn - \widehat{m} \otimes \widehat{m}$ and where $ \widehat{m} = \frac{1}{n}\sum_{i=1}^n K_{x_i}$. $\widehat{C}$~is the empirical covariance operator and $\widehat{\D}_\lambda = \widehat{\D} + \lambda I$ is a regularized empirical version of the operator $\D$ restricted to $\left(\mathrm{Ker} (\Delta) \right)^\perp$ as in Proposition~\ref{prop:Spectral_characterization}. 
\end{defi}
Note that regularization is necessary as the nullspace of $\widehat{\Delta}$ is no longer included in the nullspace of $\widehat{C}$ so that the Poincaré constant estimates blows up when $\lambda \to 0$. The problem in Equation \eqref{eq:Empirical_rayleigh_ratio} has a natural interpretation in terms of Poincaré inequality as it corresponds to a regularized \textbf{(PI)} for the empirical measure $\widehat{\mu}_n = \frac{1}{n} \sum_{i=1}^n \delta_{x_i}$ associated with the i.i.d. samples $x_1,\hdots,x_n$ from $d\mu$. To alleviate the notation, we will simply denote the estimator by $\widehat{\Poinca}_\mu$ until the end of the paper.

\subsection{Statistical consistency of the estimator}

We show that, under some assumptions and by choosing carefully $\lambda$ as a function of $n$, the estimator~$\widehat{\Poinca}_\mu$ is statistically consistent, i.e., almost surely: 
$$\widehat{\Poinca}_\mu \xrightarrow{n \rightarrow \infty} \Poinca_\mu.$$
As we regularized our problem, we prove the convergence in two steps: first, the convergence of $\widehat{\Poinca}_\mu$ to the regularized problem $\Poinca^\lambda_\mu = \sup_{f \in \h \setminus \{0\}} \frac{\langle f,C f \rangle}{\langle f, (\D + \lambda I) f \rangle } = \|\D_\lambda^{-1/2} C \D_\lambda^{-1/2}\|$, which corresponds to controlling the statistical error associated with the estimator $\widehat{\Poinca}_\mu$ (variance); second, the convergence of $\Poinca^\lambda_\mu$ to $\Poinca_\mu$ as $\lambda$ goes to zero which corresponds to the bias associated with the estimator $\widehat{\Poinca}_\mu$. The next result states the statistical consistency of the estimator when $\lambda$ is a sequence going to zero as $n$ goes to infinity (typically as an inverse power of $n$). 

\begin{thm}[Statistical consistency]
\label{thm:statistical_consistency}
Assume that \asm{asm:density}, \asm{asm:regularity_RKHS}, \asm{asm:bounded_kernel} hold true and that the operator $\Delta^{-1/2} C \Delta^{-1/2}$ is compact on $\h$. Let $(\lambda_n)_{n \in \N}$ be a sequence of positive numbers such that $\lambda_n \rightarrow 0$ and $\lambda_n\sqrt{n} \rightarrow + \infty$. Then, almost surely,$$\widehat{\Poinca}_\mu \xrightarrow{n \rightarrow \infty} \Poinca_\mu.$$
\end{thm}

As already mentioned, the proof is divided into two steps: the analysis of the statistical error for which we have an explicit rate of convergence in probability (see Proposition~\ref{prop:hat_P_to_P_lambda} below) and which requires $n^{-1/2}/\lambda_n \rightarrow 0$, and the analysis of the bias for which we need $\lambda_n \rightarrow 0$ and the compactness condition (see Proposition \ref{prop:P_lambda_to_P}). Notice that the compactness assumption in Proposition~\ref{prop:P_lambda_to_P} and Theorem~\ref{thm:statistical_consistency} is stronger than \textbf{(PI)}. Indeed, it can be shown that satisfying \textbf{(PI)} is equivalent to having the operator $\Delta^{-1/2} C \Delta^{-1/2}$ bounded whereas to have convergence of the bias we need compactness. Note also that $\lambda_n = n^{-1/4}$ matches the two conditions stated in Theorem \ref{thm:statistical_consistency} and is the optimal balance between the rate of convergence of the statistical error (of order $\frac{1}{\lambda \sqrt{n}}$, see Proposition \ref{prop:hat_P_to_P_lambda}) and of the bias we obtain in some cases (of order~$\lambda$, see Section \ref{sec:analysis_of_bias} of the Appendix). Note that the rates of convergence do not depend on the dimension $d$ of the problem which is a usual strength of kernel methods and differ from local methods like diffusion maps \cite{COIFMAN20065,hein2007graph}.

For the statistical error term, it is possible to quantify the rate of convergence of the estimator to the regularized Poincaré constant as shown below.

\begin{prop}[Analysis of the statistical error]
\label{prop:hat_P_to_P_lambda}
Suppose that \asm{asm:density}, \asm{asm:regularity_RKHS}, \asm{asm:bounded_kernel} hold true. For any $\delta \in (0,1/3)$, and $\lambda > 0$ such that $\lambda \leqslant\|\Delta\| $ and any integer $n \geqslant 15 \frac{\mathcal{K}_d}{\lambda} \log \frac{4\,\mathrm{Tr } \D }{\lambda \delta }$, with probability at least $1-3\delta$,
\begin{align}
\label{eq:hat_P_to_P_lambda}
\left|\widehat{\Poinca}_\mu-\Poinca^\lambda_\mu\right| \leqslant \frac{8 \mathcal{K}}{\lambda \sqrt{n}}\log (2/\delta) + \mathrm{o}\left(\frac{1}{\lambda \sqrt{n}}\right).
\end{align}
\end{prop}

Note that in Proposition~\ref{prop:hat_P_to_P_lambda} we are only interested in the regime where $\lambda \sqrt{n}$ is large. Lemmas \ref{lemma:concentration_C} and \ref{lemma:concentration_delta} of the Appendix give explicit and sharper bounds under refined hypotheses on the spectra of $C$ and $\Delta$. Recall also that under assumption \asm{asm:bounded_kernel}, $C$ and $\Delta$ are trace-class operators (as proved in the Appendix, Section \ref{subsec:operators}) so that $\|\Delta\|$ and $\tr(\Delta)$ are indeed finite. Finally, remark that \eqref{eq:hat_P_to_P_lambda} implies the almost sure convergence of the statistical error by applying the Borel-Cantelli lemma.

\begin{prop}[Analysis of the bias]
\label{prop:P_lambda_to_P}
Assume that \asm{asm:density}, \asm{asm:regularity_RKHS}, \asm{asm:bounded_kernel} hold true, and that the bounded operator $\Delta^{-1/2} C \Delta^{-1/2}$ is compact on $\h$. Then,
\begin{align*}
\underset{\lambda \rightarrow 0}{\lim} \ \Poinca^\lambda_\mu = \Poinca.
\end{align*}
\end{prop}

As said above the compactness condition (similar to the one used for convergence proofs of kernel Canonical Correlation Analysis \citep{fukumizu2007statistical}) is stronger than satisfying \textbf{(PI)}. The compactness condition adds conditions on the spectrum of $\Delta^{-1/2} C \Delta^{-1/2}$: it is discrete and accumulates at~$0$. We give more details on this condition in Section \ref{sec:analysis_of_bias} of the Appendix and derive explicit rates of convergence under general conditions. We derive also a rate of convergence for more specific structures (Gaussian case or under an assumption on the support of $\mu$) in Sections \ref{sec:analysis_of_bias} and \ref{sec:gaussian_bias} of the Appendix.

\section{Learning a Reaction Coordinate}
\label{sec:learning_RC}

If the measure $\mu$ is multimodal, the Langevin dynamics~\eqref{eq:langevin} is trapped for long times in certain regions (modes) preventing it from efficient space exploration. This phenomenon is called \textit{metastability} and is responsible for the slow convergence of the diffusion to its equilibrium \citep{Lelievre2013,Lelievre_2008}. Some efforts in the past decade \citep{Lelievre2015} have focused on understanding this multimodality by capturing the behavior of the dynamics at a coarse-grained level, which often have a low-dimensional nature. The aim of this section is to take advantage of the estimation of the Poincaré constant to give a procedure to unravel these dynamically meaningful slow variables called reaction coordinate. 

\subsection{Good Reaction Coordinate}
\label{subsec:GRC}
From a numerical viewpoint, a good reaction coordinate can be defined as a low dimensional function $\xi: \R^d \to \R^p\  (p \ll d) $ such that the family of conditional measures $\left(\mu(\cdot | \xi(x) = r)\right)_{z \in \R^p}$ are ``less multimodal'' than the measure $d\mu$. This can be fully formalized in particular in the context of free energy techniques such as the adaptive biasing force method, see for example \cite{Lelievre_2008}. For more details on mathematical formalizations of metastability, we also refer to \cite{Lelievre2013}. The point of view we will follow in this work is to choose $\xi$ in order to maximize the Poincaré constant of the pushforward distribution $\xi * \mu$. The idea is to capture in $\xi * \mu$ the essential multimodality of the original measure, in the spirit of the two scale decomposition of Poincaré or logarithmic Sobolev constant inequalities \citep{Lelievre2009,menz2014,OTTO2007121}.

\subsection{Learning a Reaction Coordinate}
\label{subsec:learning_RC}

\paragraph{Optimization problem.} Let us assume in this subsection that the reaction coordinate is an orthogonal projection onto a linear subspace of dimension $p$. Hence $\xi$ can be represented by $\forall x \in \R^d, \  \xi(x) = A x$ with $A \in \mathcal{S}^{p,d}$ where $\mathcal{S}^{p,d} = \{A \in \R^{p\times d}\ \mathrm{s.\ t.}\ A A^\top = I_p \}$ is the Stiefel manifold~\citep{edelman1998geometry}. As discussed in Section~\ref{subsec:GRC}, to find a good reaction coordinate we look for $\xi$ for which the Poincaré constant of the pushforward measure $\xi * \mu$ is the largest. Given $n$ samples, let us define the matrix $X = (x_1,\hdots,x_n)^\top \in \R^{n \times d}$. We denote by $\widehat{\Poinca}_{X}$ the estimator of the Poincaré constant using the samples $(x_1,\hdots,x_n)$. Hence $\widehat{\Poinca}_{AX^\top}$ defines an estimator of the Poincaré constant of the pushforward measure $\xi * \mu$. Our aim is to find $\underset{A \in \mathcal{S}^{p,d}}{\mathrm{argmax}}\ \ \widehat{\Poinca}_{AX^\top}$.

\paragraph{Random features.} One computational issue with the estimation of the Poincaré constant is that building $\widehat{C}$ and $\widehat{\Delta}$ requires respectively constructing $n \times n$ and $nd \times nd$ matrices. Random features~\citep{Rahimi2008} avoid this problem by building explicitly features that approximate a translation invariant kernel $\ K(x,x') = K(x-x')$. More precisely, let $M$ be the number of random features, $(w_m)_{1 \leqslant m\leqslant M}$ be random variables independently and identically distributed according to $\P(dw) = \int_{\R^d} \mathrm{e}^{-\mathrm{i} w^\top \delta} K(\delta) d \delta \, dw$ and $(b_m)_{1 \leqslant m\leqslant M}$ be independently and identically distributed according to the uniform law on $[0,2\pi]$, then the feature vector $\phi^M (x) = \sqrt{\frac{2}{M}} \left(\cos(w_1^\top x + b_1), \hdots, \cos(w_M^\top x + b_M)\right)^\top \in \R^M$ satisfies $K(x,x') \approx \phi^M (x)^\top \phi^M (x') $. Therefore, random features allow to approximate $\widehat{C}$ and $\widehat{\Delta}$ by $M \times M$ matrices $\widehat{C}^M$ and $\widehat{\D}^M$ respectively. Finally, when these matrices are constructed using the projected samples, i.e.~$\left(\cos(w_m^\top A x_i + b_m)\right)_{_{{ \substack{1\leq m\leq M \\ 1\leq i\leq n}}}}$ , we denote them by $\widehat{C}^M_A$ and $\widehat{\D}^M_A$ respectively. Hence, the problem reads
\begin{align}
\label{eq:optimization_problem}
\mathrm{Find} \  \underset{A \in \mathcal{S}^{p,d}}{\mathrm{argmax}}\ \ \widehat{\Poinca}_{AX^\top} = \underset{A \in \mathcal{S}^{p,d}}{\mathrm{argmax}}\  \max_{v \in \R^M \setminus \{0\}} \ F(A,v)\ ,\quad \textrm{where}\ F(A,v):=\frac{v^\top \widehat{C}^M_A v}{v^\top (\widehat{\D}^M_A + \lambda I) v }.
\end{align}

\paragraph{Algorithm.} To solve the non-concave optimization problem \eqref{eq:optimization_problem}, our procedure is to do one step of non-Euclidean gradient descent to update $A$ (gradient descent in the Stiefel manifold) and one step by solving the generalized eigenvalue problem to update $v$. More precisely, the algorithm~reads:

\vspace*{0.25cm}
\begin{algorithm}[H]
 \KwResult{Best linear Reaction Coordinate: $A_* \in \mathcal{S}^{d, p}$}
 $A_0$ random matrix in $\mathcal{S}^{d, p}$, $\eta_t > 0$ step-size\;
 \For{$t = 0, \hdots, T-1$}{
 \begin{itemize}
  \item Solve generalized largest eigenvalue problem with matrices $\widehat{C}^M_{A_t}$ and $\widehat{\D}^M_{A_t}$ to get~$v^* (A_t)$: $$ v^* (A_t) = \underset{v \in \R^M \setminus\{0\}}{\mathrm{argmax}}\ \  \frac{v^\top \widehat{C}^M_A v}{v^\top (\widehat{\D}^M_A + \lambda I) v }. $$
  \item Do one gradient ascent step: $A_{t+1} = A_t + \eta_t\ \mathrm{grad}_A\, F(A,v^* (A_t)).$
 \end{itemize}
  }
 \caption{Algorithm to find best linear Reaction Coordinate.}
\end{algorithm}

\section{Numerical experiments}
\label{sec:experiments}

We divide our experiments into two parts: the first one illustrates the convergence of the estimated Poincaré constant as given by Theorem~\ref{thm:statistical_consistency} (see Section~\ref{subsec:experiments_poincare}), and the second one demonstrates the interest of the reaction coordinates learning procedure described in Section~\ref{subsec:learning_RC} (see Section~\ref{subsec:experiments_RC}).

\subsection{Estimation of the Poincaré constant}
\label{subsec:experiments_poincare}

In our experiments we choose the Gaussian Kernel $K(x,x') = \exp\, (-\|x-x'\|^2)$? This induces a RKHS satisfying \asm{asm:density}, \asm{asm:regularity_RKHS}, \asm{asm:bounded_kernel}. Estimating $\widehat{\Poinca}_\mu$ from $n$ samples $(x_i)_{i \leqslant n}$ is equivalent to finding the largest eigenvalue for an operator from $\h$ to $\h$. Indeed, we have 
$$ \widehat{\Poinca}_\mu = \left\| (\widehat{Z}_n^* \widehat{Z}_n + \lambda I)^{-\frac{1}{2}} \widehat{S}_n^* \left(I - \frac{1}{n}\mathds{1}\mathds{1}^\top\right)\, \widehat{S}_n\, (\widehat{Z}_n^* \widehat{Z}_n + \lambda I)^{-\frac{1}{2}} \right\|_\h, $$
where $\widehat{Z}_n = \sum_{i=1}^d \widehat{Z}^i_n$ and $\widehat{Z}^i_n$ is the operator from $\h$ to $\R^n$: $\forall g\in \h$, $\widehat{Z}^i_n(g) = \frac{1}{\sqrt{n}} \left( \langle  g, \partial_i K_{x_j} \rangle \right)_{1 \leqslant j\leqslant n}$ and $\widehat{S}_n$ is the operator from $\h$ to $\R^n$: $\forall g\in \h$, $\widehat{S}_n(g) = \frac{1}{\sqrt{n}} \left( \langle  g, K_{x_j} \rangle \right)_{1 \leqslant j\leqslant n}$. 
By the Woodbury operator identity, $(\lambda I+\widehat{Z}_n^*\widehat{Z}_n)^{-1}=\frac{1}{\lambda}\left(I- \widehat{Z}_n^*(\lambda I+\widehat{Z}_n\widehat{Z}_n^*)^{-1}\widehat{Z}_n\right)$, and the fact that for any operator $\|T^* T\| = \|T T^*\|$,
\begin{align*}
\widehat{\Poinca}_\mu &= \left\| (\widehat{Z}_n^* \widehat{Z}_n + \lambda I)^{-\frac{1}{2}} \widehat{S}_n^* \left(I - \frac{1}{n}\mathds{1}\mathds{1}^\top\right)\, \widehat{S}_n\, (\widehat{Z}_n^* \widehat{Z}_n + \lambda I)^{-\frac{1}{2}} \right\|_\h \\
&= \left\| (\widehat{Z}_n^* \widehat{Z}_n + \lambda I)^{-\frac{1}{2}} \widehat{S}_n^* \left(I - \frac{1}{n}\mathds{1}\mathds{1}^\top\right) \left(I - \frac{1}{n}\mathds{1}\mathds{1}^\top\right)\, \widehat{S}_n\, (\widehat{Z}_n^* \widehat{Z}_n + \lambda I)^{-\frac{1}{2}} \right\|_\h \\
&= \left\| \left(I - \frac{1}{n}\mathds{1}\mathds{1}^\top\right)\, \widehat{S}_n\, (\widehat{Z}_n^* \widehat{Z}_n + \lambda I)^{-1} \widehat{S}_n^* \left(I - \frac{1}{n}\mathds{1}\mathds{1}^\top\right) \right\|_2 \\
&= \frac{1}{\lambda} \left\| \left(I - \frac{1}{n}\mathds{1}\mathds{1}^\top\right)  \left(\widehat{S}_n \widehat{S}_n^* - \widehat{S}_n \widehat{Z}_n^* \, (\widehat{Z}_n \widehat{Z}_n^* + \lambda I)^{-1} \widehat{Z}_n \widehat{S}_n^*\right) \left(I - \frac{1}{n}\mathds{1}\mathds{1}^\top\right) \right\|_2,
\end{align*}
which is now the largest eigenvalue of a $n \times n$ matrix built as the product of matrices involving the kernel $K$ and its derivatives. Note for the above calculation that we used that $\left(I - \frac{1}{n}\mathds{1}\mathds{1}^\top\right)^2 = \left(I - \frac{1}{n}\mathds{1}\mathds{1}^\top\right)$.

\begin{figure}[ht]
\includegraphics[width=0.49\textwidth]{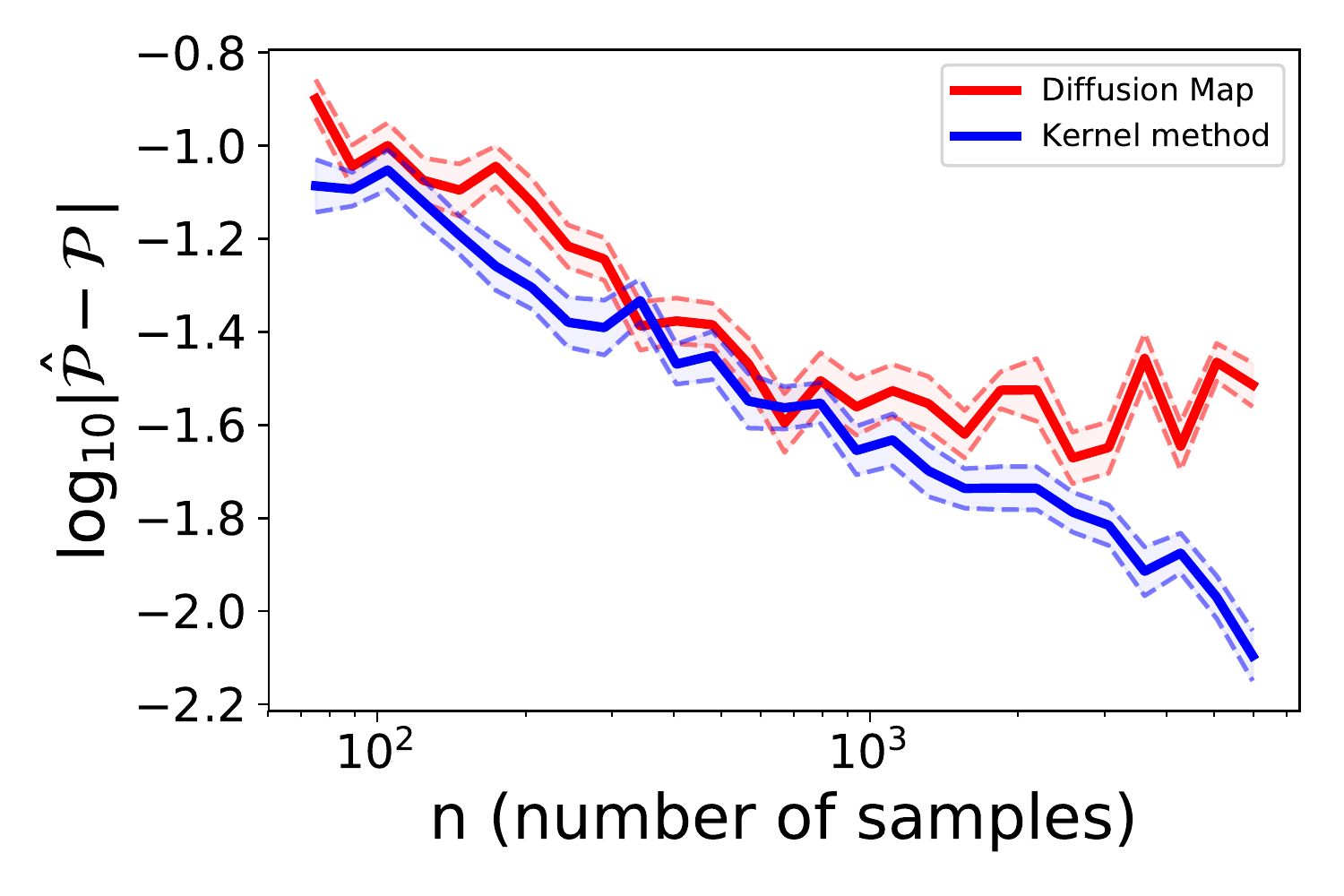}
\hspace{0.1cm}%
\includegraphics[width=0.49\textwidth]{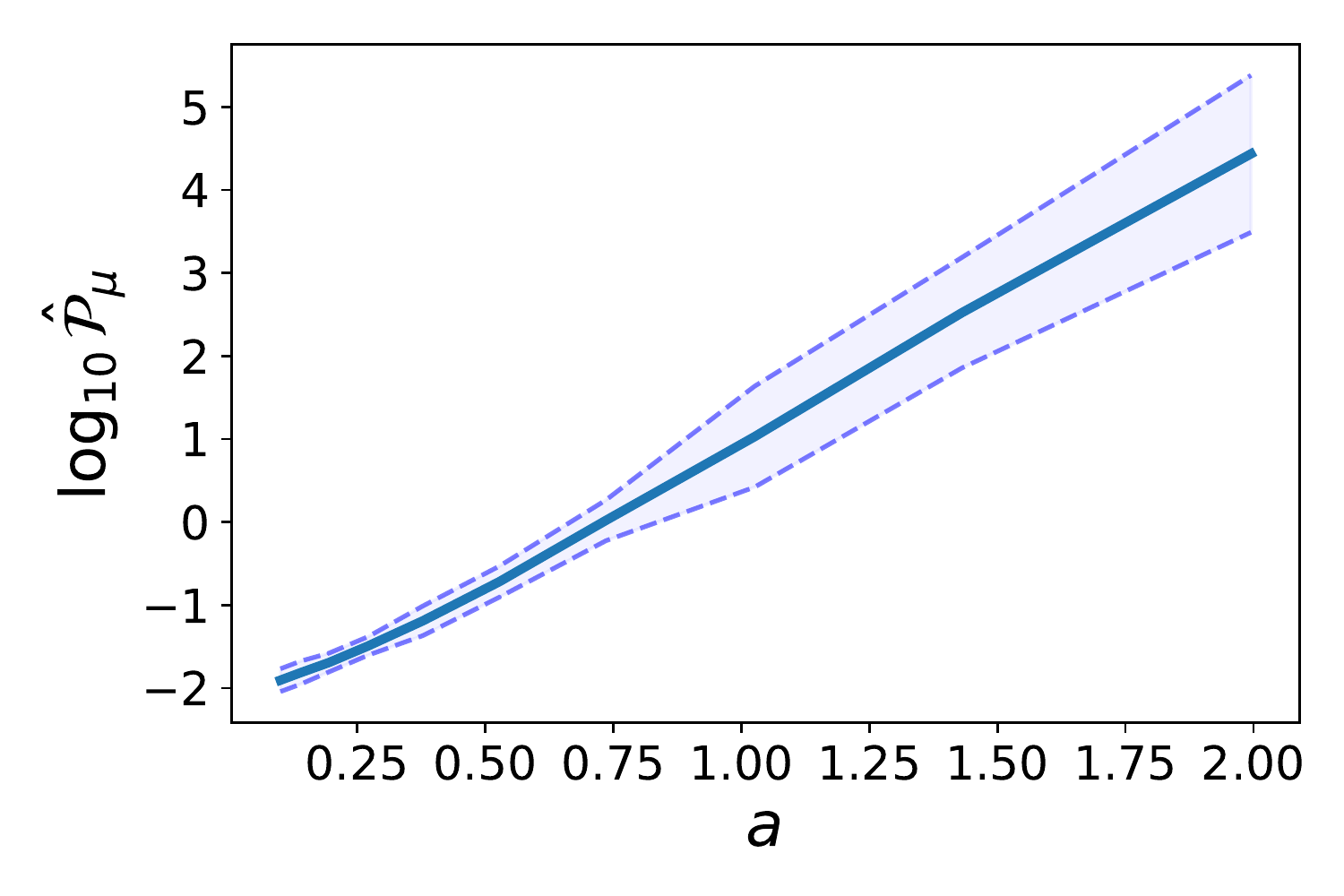} 
\caption{ \textbf{(Left)} Comparison of the convergences of the kernel-based method described in this paper and diffusion maps in the case of a Gaussian of variance~$1$ (for each $n$ we took the mean over $50$ runs). The dotted lines correspond to standard deviations of the estimator. \textbf{(Right)} Exponential growth of the Poincaré constant for a mixture of two Gaussians $\mathcal{N}(\pm \frac{a}{2},\sigma^2)$ as a function of the distance $a$ between the two Gaussians ($\sigma = 0.1$ and $n = 500$).}
\label{fig:poinca}
\end{figure}

We illustrate in Figure~\ref{fig:poinca} the rate of convergence of the estimated Poincaré constant to $1$ for the Gaussian $\mathcal{N}(0,1)$ as the number of samples $n$ grows. Recall that in this case the Poincaré constant is equal to~$1$ (see Subsection \ref{subsec:examples}). We compare our prediction to the one given by diffusion maps techniques \citep{COIFMAN20065}. For our method, in all the experiments we set $\lambda_n = \frac{C_\lambda}{n}$, which is smaller than what is given by Theorem \ref{thm:statistical_consistency}, and optimize the constant $C_\lambda$ with a grid search. Following \citep{hein2007graph}, to find the correct bandwidth $\varepsilon_n$ of the kernel involved in diffusion maps, we performed a similar grid search on the constant $C_\varepsilon$ for the Diffusion maps with the scaling $\varepsilon_n = \frac{C_\varepsilon}{n^{1/4}}$. Additionally to a faster convergence when $n$ become large, the kernel-based method is more robust with respect to the choice of itss hyperparameter, which is of crucial importance for the quality of diffusion maps. Note also that we derive an explicit convergence rate for the bias in the Gaussian case in Section \ref{sec:gaussian_bias} of the Appendix. In Figure~\ref{fig:poinca}, we also show the growth of the Poincaré constant for a mixture of Gaussians of variances~$1$ as a function of the distance  between the two means of the Gaussians. This is a situation for which the estimation provides an estimate when, up to our knowledge, no precise Poincaré constant is known (even if lower and upper bounds are known \citep{Chafai2010}). 

\subsection{Learning a reaction coordinate}
\label{subsec:experiments_RC}

We next illustrate the algorithm described in Section \ref{sec:learning_RC} to learn a reaction coordinate which, we recall, encodes directions which are difficult to sample. To perform the gradient step over the Stiefel manifold we used Pymanopt \citep{Pymanopt2016}, a Python library for manifold optimization derived from Manopt~\citep{manopt} (Matlab). We show here a synthetic two-dimensional example example. We first preprocessed the samples with ``whitening'', i.e., making it of variance~$1$ in all directions to avoid scaling artifacts. In both examples, we took $M = 200 $ for the number of random features and $n = 200 $ for the number of samples.

\begin{figure}[ht]
\footnotesize
\includegraphics[width=0.49\textwidth]{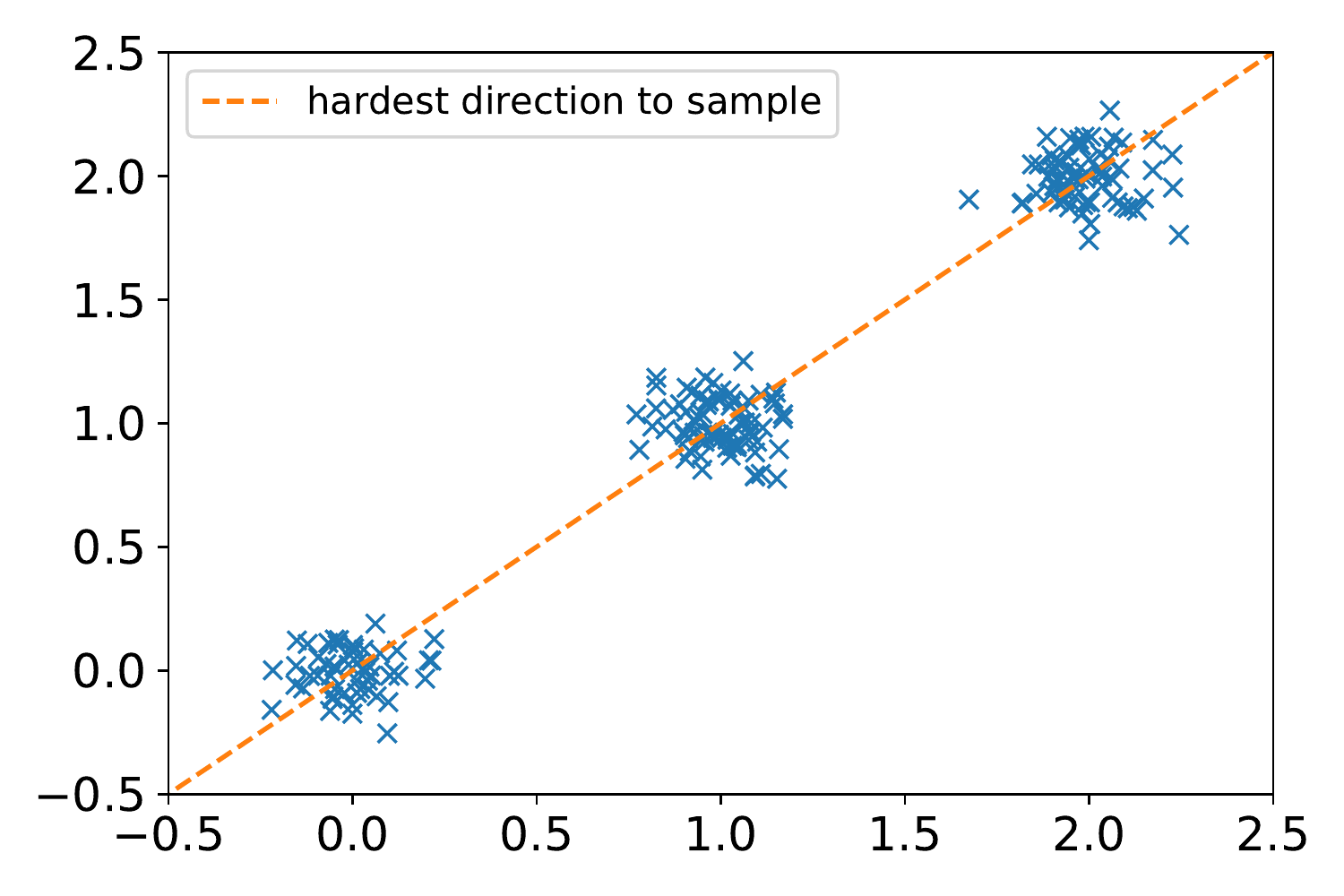}
\hspace{0.1cm}%
\includegraphics[width=0.49\textwidth]{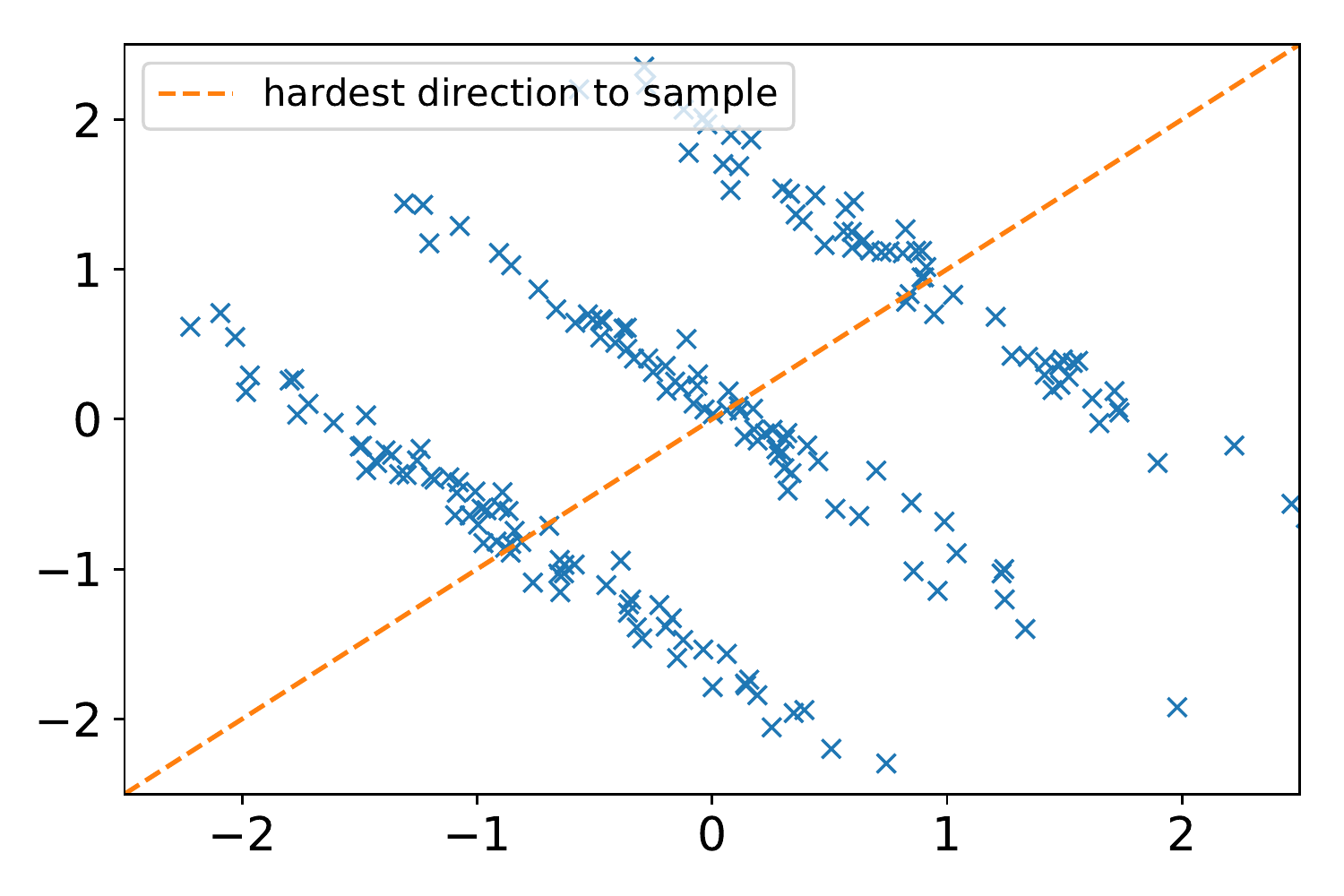}
\\
\begin{center}
\includegraphics[width=0.49\textwidth]{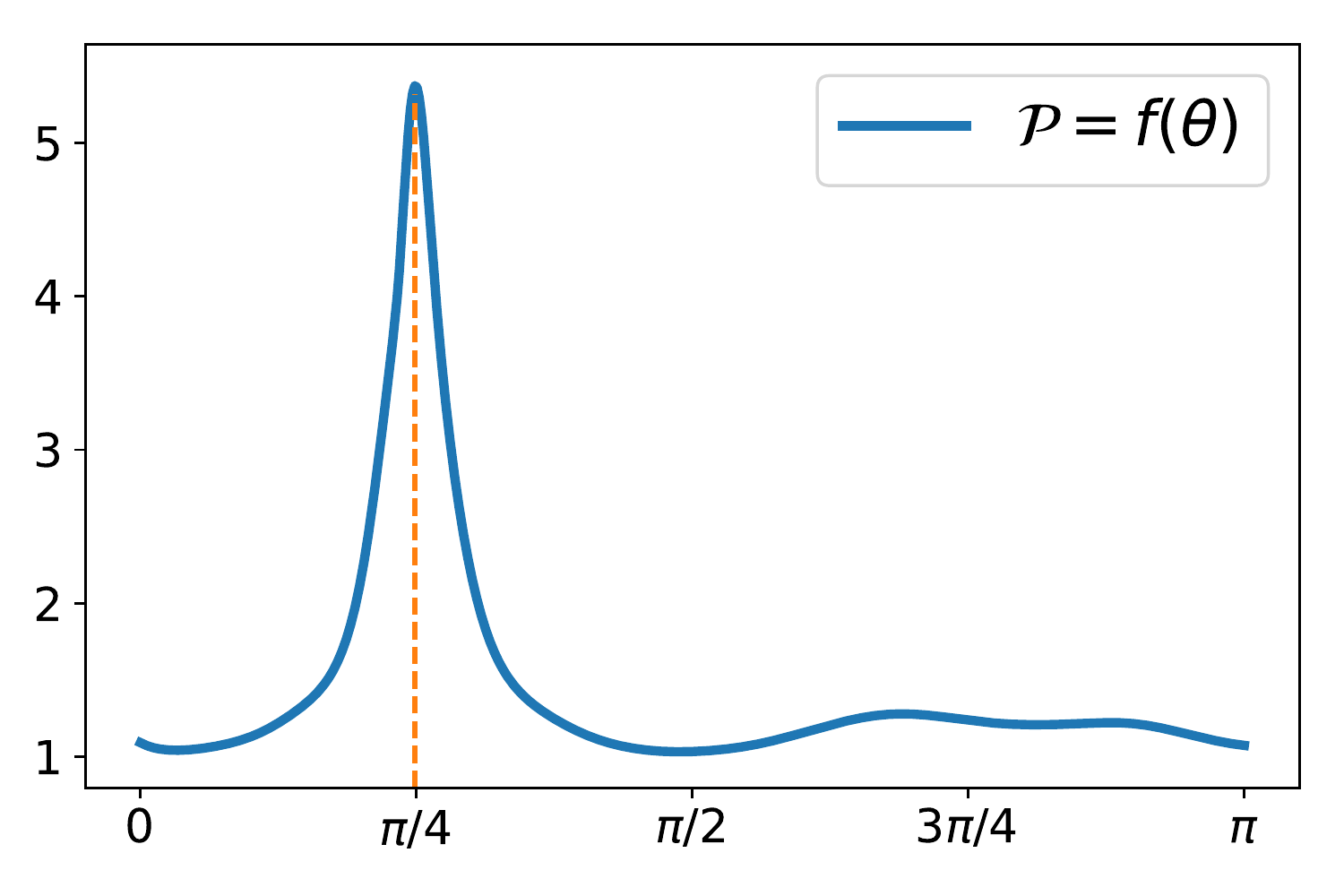}
\end{center}

\caption{ \textbf{(Top Left)} Samples of mixture of three Gaussians. \textbf{(Top right)} Whiten samples of Gaussian mixture on the left. \textbf{(Bottom)} Plot of the Poincaré constant of the projected samples on a line of angle $\theta$.}
\label{fig:three_gaussians}
\end{figure}

We show (Figure~\ref{fig:three_gaussians}) one synthetic example for which our algorithm found a good reaction coordinate. The samples are taken from a mixture of three Gaussians of means $(0,0), (1,1) $ and $(2,2)$ and covariance $\Sigma = \sigma^2 I$ where $\sigma = 0.1$. The three means are aligned along a line which makes an angle $\theta = \pi/4$ with respect to the $x$-axis: one expects the algorithm to identify this direction as the most difficult one to sample (see left and center plots of Figure~\ref{fig:three_gaussians}). With a few restarts, our algorithm indeed finds the largest Poincaré constant for a projection onto the line parametrized by $\theta = \pi / 4$.


\section{Conclusion and Perspectives}

In this paper, we have presented an efficient method to estimate the Poincaré constant of a distribution from independent samples, paving the way to learn low-dimensional marginals that are hard to sample (corresponding to the image measure of so-called reaction coordinates). While we have focused on linear projections, learning non-linear projections is important in molecular dynamics and it can readily be done with a well-defined parametrization of the non-linear function and then applied to real data sets, where this would lead to accelerated sampling~\citep{Lelievre2015}. Finally, it would be interesting to apply our framework to Bayesian inference~\citep{Chopin2012} and leverage the knowledge of reaction coordinates to accelerate sampling methods.

\bibliographystyle{plain}

\bibliography{Poincare_estimation.bib}

\clearpage

\onecolumn 

\appendix

\textbf{\huge Appendix}

$ $

\vspace{0.2cm}

The Appendix is organized as follows. In Section \ref{sec:proofs_of_prop_1/2} we prove Propositions~\ref{prop:Spectral_characterization} and~\ref{prop:hat_P_to_P_lambda}. Section~\ref{sec:analysis_of_bias} is devoted to the analysis of the bias. We study spectral properties of the diffusion operator $L$ to give sufficient and general conditions for the compactness assumption from Theorem~\ref{thm:statistical_consistency} and Proposition~\ref{prop:P_lambda_to_P} to hold. Section~\ref{sec:technical_inequalities} provides concentration inequalities for the operators involved in Proposition~\ref{prop:hat_P_to_P_lambda}. We conclude by Section~\ref{sec:gaussian_bias} that gives explicit rates of convergence for the bias when~$\mu$ is a 1-D Gaussian (this result could be easily extended to higher dimensional Gaussians).

\section{Proofs of Proposition \ref{prop:Spectral_characterization} and \ref{prop:hat_P_to_P_lambda}}
\label{sec:proofs_of_prop_1/2}

Recall that $L_0^2(\mu)$ is the subspace of $L^2(\mu)$ of zero mean functions: $L_0^2(\mu):= \{ f \in L^2(\mu),\ \int f(x)d\mu(x) = 0 \}$ and that we similarly defined $\h_0 := \h \cap L_0^2(\mu)$. Let us also denote by $\R \mathds{1}$  , the set of constant functions.

\begin{proof}[Proof of Proposition \ref{prop:Spectral_characterization}]

The proof is simply the following reformulation of Equation~\eqref{eq:Poicare_constant}. Under assumption \asm{asm:density}:
\begin{align*}
\Poinca_\mu &= \sup_{f \in H^1(\mu) \setminus \R \mathds{1}} \frac{\int_{\R^d}f(x)^2 d\mu(x) - \left(\int_{\R^d}f(x) d\mu(x)\right)^2 }{\int_{\R^d}\|\nabla f(x)\|^2 d\mu(x)} \\
&=  \sup_{f \in \h \setminus \R \mathds{1}} \frac{\int_{\R^d}f(x)^2 d\mu(x) - \left(\int_{\R^d}f(x) d\mu(x)\right)^2 }{\int_{\R^d}\|\nabla f(x)\|^2 d\mu(x)} \\
&=  \sup_{f \in \h_0 \setminus \{0\}} \frac{\int_{\R^d}f(x)^2 d\mu(x) - \left(\int_{\R^d}f(x) d\mu(x)\right)^2 }{\int_{\R^d}\|\nabla f(x)\|^2 d\mu(x)}.
\end{align*}
We then simply note that
\begin{align*}
\left(\int_{\R^d}f(x) d\mu(x)\right)^2 &= \left(\left\langle f,\int_{\R^d}K_x d\mu(x) \right\rangle_\h \right)^2 = \langle f,m \rangle_\h^2 = \langle f,(m \otimes m) f \rangle_\h.
\end{align*}
Similarly,
\begin{align*}
\int_{\R^d}f(x)^2 d\mu(x) &= \langle f,\C f \rangle_\h \quad \textrm{and} \quad  \int_{\R^d}\|\nabla f(x)\|^2 d\mu(x) = \langle f,\D f \rangle_\h.
\end{align*}
Note here that $\mathrm{Ker} (\Delta) \subset \mathrm{Ker} (C)$. Indeed, if $f \in \mathrm{Ker} (\Delta)$, then $\langle f, \Delta f \rangle_\h = 0$. Hence, $\mu$-almost everywhere, $\nabla f = 0$ so that $f$ is constant and $C f = 0$. Note also the previous reasoning shows that $\mathrm{Ker} (\Delta) $ is the subset of $\h$ made of constant functions, and $(\mathrm{Ker} (\Delta))^\perp = \h \cap L^2_0(\mu)= \h_0$.

Thus we can write,
\begin{align*}
\Poinca_\mu &= \sup_{f \in \h \setminus \mathrm{Ker} (\Delta)} \frac{\langle f,(\C - m \otimes m) f \rangle_\h}{\langle f,\D f \rangle_\h} = \left\|\D^{-1/2} C \D^{-1/2}\right\|,
\end{align*}
where we consider $\Delta^{-1}$ as the inverse of $\Delta$ restricted to $\left(\mathrm{Ker} (\Delta) \right)^\perp$ and thus get Proposition \ref{prop:Spectral_characterization}.
\end{proof}

\begin{proof}[Proof of Proposition \ref{prop:hat_P_to_P_lambda}] 

We refer to Lemmas \ref{lemma:concentration_C} and \ref{lemma:concentration_delta} in Section~\ref{sec:technical_inequalities} for the explicit bounds. We have the following inequalities:

\begin{align*}
\left|\widehat{\Poinca}_\mu-\Poinca^\lambda_\mu\right| &= \left|\left\|\widehat{\D}_\lambda^{-1/2} \widehat{C} \widehat{\D}_\lambda^{-1/2}\right\| - \left\|\D_\lambda^{-1/2} C \D_\lambda^{-1/2}\right\|\right| \\
&\leqslant \left|\left\|\widehat{\D}_\lambda^{-1/2} \widehat{C} \widehat{\D}_\lambda^{-1/2}\right\| - \left\|\widehat{\D}_\lambda^{-1/2} C \widehat{\D}_\lambda^{-1/2}\right\|\right| + \left|\left\|\widehat{\D}_\lambda^{-1/2} C \widehat{\D}_\lambda^{-1/2}\right\| -  \left\|\D_\lambda^{-1/2} C \D_\lambda^{-1/2}\right\|\right| \\
&\leqslant \left\|\widehat{\D}_\lambda^{-1/2} (\widehat{C}-C) \widehat{\D}_\lambda^{-1/2}\right\| + \left|\left\|C^{1/2} \widehat{\D}_\lambda^{-1} C^{1/2} \right\| -  \left\|C^{1/2} \D_\lambda^{-1} C^{1/2} \right\|\right| \\
&\leqslant \left\|\widehat{\D}_\lambda^{-1/2} (\widehat{C}-C) \widehat{\D}_\lambda^{-1/2}\right\| + \left\|C^{1/2} (\widehat{\D}_\lambda^{-1}-\D_\lambda^{-1}) C^{1/2} \right\|.
\end{align*}
Consider an event where the estimates of Lemmas \ref{lemma:concentration_C}, \ref{lemma:concentration_delta} and \ref{lemma:magic_2} hold for a given value of $\delta > 0$. A simple computation shows that this event has a probability $1-3\delta$ at least. We study the two terms above separately. First, provided that $n \geqslant 15 \Fla \log \frac{4\,\mathrm{Tr } \D }{\lambda \delta }$ and $\lambda \in (0, \|\Delta\|]$ in order to use Lemmas~\ref{lemma:concentration_delta} and \ref{lemma:magic_2},
\begin{align*}
\left\|\widehat{\D}_\lambda^{-1/2} (\widehat{C}-C) \widehat{\D}_\lambda^{-1/2}\right\| &= \left\|\widehat{\D}_\lambda^{-1/2} \D_\lambda^{1/2} \D_\lambda^{-1/2} (\widehat{C}-C) \D_\lambda^{-1/2} \D_\lambda^{1/2} \widehat{\D}_\lambda^{-1/2}\right\| \\
&\leqslant \underbrace{\left\|\widehat{\D}_\lambda^{-1/2} \D_\lambda^{1/2} \right\|^2}_{\textrm{Lemma}\ \ref{lemma:magic_2}} \underbrace{\left\|\D_\lambda^{-1/2}  (\widehat{C}-C) \D_\lambda^{-1/2}\right\|}_{\textrm{Lemma}\ \ref{lemma:concentration_C}} \\
&\leqslant 2 \, (\textrm{Lemma}\ \ref{lemma:concentration_C}).
\end{align*}
For the second term,
\begin{align*}
\left\|C^{1/2} (\widehat{\D}_\lambda^{-1}-\D_\lambda^{-1}) C^{1/2} \right\| &= \left\|C^{1/2} \widehat{\D}_\lambda^{-1}  (\D-\widehat{\D}) \D_\lambda^{-1} C^{1/2} \right\| \\
&= \left\|C^{1/2} \D_\lambda^{-1/2} \D_\lambda^{1/2} \widehat{\D}_\lambda^{-1} \D_\lambda^{1/2} \D_\lambda^{-1/2}  (\D-\widehat{\D}) \D_\lambda^{-1/2} \D_\lambda^{-1/2} C^{1/2} \right\| \\
&\leqslant \underbrace{\left\|\widehat{\D}_\lambda^{-1/2} \D_\lambda^{1/2} \right\|^2}_{\textrm{Lemma}\ \ref{lemma:magic_2}}  \underbrace{\left\|C^{1/2} \D_\lambda^{-1/2}\right\|^2}_{\Poinca_\mu^\lambda } \underbrace{\left\| \D_\lambda^{-1/2}  (\D-\widehat{\D}) \D_\lambda^{-1/2} \right\|}_{\textrm{Lemma}\ \ref{lemma:concentration_delta}} \\
&\leqslant 2\cdot \Poinca^\lambda_\mu \cdot(\textrm{Lemma}\ \ref{lemma:concentration_delta}). 
\end{align*}
The leading order term in the estimate of Lemma~\ref{lemma:concentration_delta} is of order $\left(\frac{2\mathcal{K}_d \log (4 \tr \D/\lambda\delta )}{\lambda n}\right)^{1/2}$  whereas the leading one in Lemma~\ref{lemma:concentration_C} is of order $\frac{8 \mathcal{K} \log (2/\delta )}{\lambda\sqrt{ n}}$. Hence, the latter is the dominant term in the final estimation.
\end{proof}
\section{Analysis of the bias: convergence of the regularized Poincaré constant to the true one}
\label{sec:analysis_of_bias}

We begin this section by proving Proposition \ref{prop:P_lambda_to_P}. We then investigate the compactness condition required in the assumptions of Proposition \ref{prop:P_lambda_to_P} by studying the spectral properties of the diffusion operator $L$. In Proposition 
\ref{prop:bias_compact}, we derive, under some general assumption on the RKHS and usual growth conditions on $V$, some convergence rate for the bias term.

\subsection{General condition for consistency: proof of Proposition \ref{prop:P_lambda_to_P}}

To prove Proposition~\ref{prop:P_lambda_to_P}, we first need a general result on operator norm convergence.

\begin{lemma}
\label{le:compact_convergence}
Let $\h$ be a Hilbert space and suppose that $(A_n)_{n \geqslant 0}$ is a family of bounded operators such that $\forall n \in \N$, $\|A_n\| \leqslant 1$ and $\forall f \in \h$, $A_n f \xrightarrow{n\to\infty} A f$. Suppose also that $B$ is a compact operator. Then, in operator norm, $$ A_n B A_n^* \xrightarrow{n\to\infty} A B A^*.$$
\end{lemma}
\begin{proof}
Let $\varepsilon > 0$. As $B$ is compact, it can be approximated by a finite rank operator $B_{_{n_\varepsilon}} = \sum_{i=1}^{n_\varepsilon} b_i \langle f_i, \cdot \rangle g_i $, where  $(f_{i})_i$  and $(g_{i})_i$  are orthonormal bases, and  $(b_{i})_i$  is a sequence of nonnegative numbers with limit zero (singular values of the operator). More precisely, $n_\varepsilon$ is chosen so that 
$$ \| B - B_{_{n_\varepsilon}} \| \leqslant \frac{\varepsilon}{2}.$$
Moreover, $\varepsilon$ being fixed, $A_n B_{_{n_\varepsilon}} A_n^* = \sum_{i=1}^{n_\varepsilon} b_i \langle A_n f_i, \cdot \rangle A_n g_i \underset{n\infty}{\longrightarrow} \sum_{i=1}^{n_\varepsilon} b_i \langle A f_i, \cdot \rangle A g_i = A B_{_{n_\varepsilon}} A^* $ in operator norm, so that, for $n \geqslant N_\varepsilon$, with $N_\varepsilon \geqslant n_\varepsilon$ sufficiently large, $\|A_n B_{_{n_\varepsilon}} A_n^* - A B_{_{n_\varepsilon}} A^*\| \leqslant \frac{\varepsilon}{2}$. Finally, as $\|A\| \leqslant 1$, it holds, for $n \geqslant N_\varepsilon$
\begin{align*}
\|A_n B_{_{n_\varepsilon}} A_n^* - A B A^*\| &\leqslant \| A_n B_{_{n_\varepsilon}} A_n^* - A B_{_{n_\varepsilon}} A^* \| + \| A ( B_{_{n_\varepsilon}} - B) A^* \| \\
& \leqslant   \| A_n B_{_{n_\varepsilon}} A_n^* - A B_{_{n_\varepsilon}} A^* \| + \| B_{_{n_\varepsilon}} - B \| \leqslant \varepsilon.
 \end{align*} 
This proves the convergence in operator norm of $A_n B A_n^*$ to $A B A^*$ when $n$ goes to infinity.
\end{proof}

We can now prove Proposition~\ref{prop:P_lambda_to_P}.

\begin{proof}[Proof of Proposition~\ref{prop:P_lambda_to_P}]
Let $\lambda > 0$, we want to show that $$\Poinca_\mu^\lambda = \|\Delta_\lambda^{-1/2}C\Delta_\lambda^{-1/2}\| \underset{\lambda \rightarrow 0}{\longrightarrow}  \|\Delta^{-1/2}C\Delta^{-1/2}\| = \Poinca_\mu. $$
Actually, with Lemma~\ref{le:compact_convergence}, we will show a stronger result which is the norm convergence of the operator $\Delta_\lambda^{-1/2}C\Delta_\lambda^{-1/2}$ to $\Delta^{-1/2}C\Delta^{-1/2}$. Indeed, denoting by $B = \Delta^{-1/2}C\Delta^{-1/2}$ and by $A_\lambda = \Delta_\lambda^{1/2}\Delta^{-1/2}$ both defined on $\h_0$, we have $\Delta_\lambda^{-1/2}C\Delta_\lambda^{-1/2} = A_\lambda B A_\lambda^*$ with $B$ compact and $\|A_\lambda\| \leqslant 1$. Furthermore, let $(\phi_i)_{i \in \N}$ be an orthonormal family of eigenvectors of the compact operator $\D$ associated to eigenvalues $(\nu_i)_{i \in \N}$. Then we can write, for any $ f \in \h_0$, $$A_\lambda f = \Delta_\lambda^{1/2}\Delta^{-1/2} f = \sum_{i=0}^\infty \sqrt{\frac{\lambda + \nu_i}{\nu_i}} \langle f, \phi_i\rangle_\h \, \phi_i \underset{\lambda \rightarrow 0}{\longrightarrow} f. $$ Hence by applying Lemma~\ref{le:compact_convergence}, we have the convergence in operator norm of $\Delta_\lambda^{-1/2}C\Delta_\lambda^{-1/2}$ to $\Delta^{-1/2}C\Delta^{-1/2}$, hence in particular the convergence of the norms of the operators.

\end{proof}

\subsection{Introduction of the operator \texorpdfstring{$L$}{L}}

\noindent In all this section we focus on a distribution $d\mu$ of the form $d\mu(x) = \mathrm{e}^{-V(x)}dx$.

Let us give first a characterization of the function that allows to recover the Poincaré constant, i.e., the function in $H^1(\mu)$ that minimizes $\frac{\int_{\R^d}\|\nabla f(x)\|^2 d\mu(x)}{\int_{\R^d}f(x)^2 d\mu(x) - \left(\int_{\R^d}f(x) d\mu(x)\right)^2 }$. We call $f_*$ this function. We recall that we denote by $\Delta^L$ the standard Laplacian in $\R^d$: $\forall f \in H^1(\mu)$, $\Delta^L f = \sum_{i=1}^d \frac{\partial^2 f_i}{\partial^2 xi}$. Let us define the operator $\forall f \in H^1(\mu)$, $Lf = - \Delta^L f + \langle \nabla V, \nabla f \rangle $, which is the opposite of the infinitesimal generator of the dynamics \eqref{eq:langevin}. We can verify that it is symmetric in $L^2(\mu)$. Indeed by integrations by parts for any $\forall f,g \in C^\infty_c$,
\begin{align*}
\langle Lf , g\rangle_{L^2(\mu)} &=  \int (Lf)(x)g(x)d\mu(x) \\
&= -\int \Delta^L f(x) g(x) \mathrm{e}^{-V(x)} dx +  \int \langle \nabla V(x), \nabla f(x) \rangle g(x) \mathrm{e}^{-V(x)} dx \\
&= \int \left\langle \nabla f(x), \nabla\left(g(x) \mathrm{e}^{-V(x)}\right) \right\rangle dx +  \int \langle \nabla V(x), \nabla f(x) \rangle g(x) \mathrm{e}^{-V(x)} dx \\
&= \int \langle \nabla f(x), \nabla g(x) \rangle \mathrm{e}^{-V(x)} dx - \int \langle \nabla f(x), \nabla V(x) \rangle g(x) \mathrm{e}^{-V(x)} dx \\
 &\hspace*{4.6cm}+  \int \langle \nabla V(x), \nabla f(x) \rangle g(x) \mathrm{e}^{-V(x)} dx \\
&= \int \langle \nabla f(x), \nabla g(x) \rangle d\mu(x).
\end{align*}
The last equality being totally symmetric in $f$ and $g$, we have the symmetry of the operator $L$:  $\langle Lf , g\rangle_{L^2(\mu)} = \int \langle \nabla f, \nabla g \rangle d\mu = \langle f , Lg\rangle_{L^2(\mu)}$ (for the self-adjointness we refer to \citep{bakry2014}). Remark that the same calculation shows that $ \nabla^* = -\mathrm{div} + \nabla V \cdot $, hence $L = \nabla^* \cdot \nabla = - \Delta^L + \langle \nabla V, \nabla \cdot \rangle$, where $\nabla^*$ is the adjoint of $\nabla$ in $L^2(\mu)$.

Let us call $\pi$ the orthogonal projector of $L^2(\mu)$ on constant functions: $ \pi f : x \in \R^d \mapsto \int f d\mu$. The problem \eqref{eq:spectral_poinca} then rewrites:
\begin{align}
\label{eq:poinca_operator}
\Poinca^{-1} = \inf_{f \in (H^1(\mu) \cap L_0^2(\mu)) \setminus \{0\}}\frac{\langle Lf,f \rangle_{L^2(\mu)}}{\|(I_{L^2(\mu)}-\pi)f\|^2},
\end{align}
Until the end of this part, to alleviate the notation we omit to mention that the scalar product is the canonical one on $L^2(\mu)$. In the same way, we also denote $\mathds{1} = I_{L^2(\mu)}$.



\subsubsection{Case where  \texorpdfstring{$d\mu$}{mu} has infinite support}

\begin{prop}[Properties of the minimizer]
\label{prop:minimizer}
If $\underset{|x| \rightarrow \infty}{\lim}\frac{1}{4} \left| \nabla V \right|^2 - \frac{1}{2} \Delta^L V = +\infty$, the problem \eqref{eq:poinca_operator} admits a minimizer in $H^1(\mu)$ and every minimizer $f$ is an eigenvector of $L$ associated with the eigenvalue $\Poinca^{-1}$:
\begin{align}
\label{eq:pde_f_star}
Lf = \Poinca^{-1} f.
\end{align} 
\end{prop}
To prove the existence of a minimizer in $H^1(\mu)$, we need the following lemmas.
\begin{lemma}[Criterion for compact embedding of $H^1(\mu)$ in $L^2(\mu)$]
\label{lemma:compact_embedding}
The injection $H^1(\mu) \hookrightarrow L^2(\mu)$ is compact if and only if the Schrödinger operator $-\Delta^L + \frac{1}{4} \left| \nabla V \right|^2 - \frac{1}{2} \Delta^L V$ has compact resolvent.
\end{lemma}

\begin{proof}
See \citep[Proposition 1.3]{gansberger2010embbeding} or \citep[Lemma XIII.65]{reed2012methods}.
\end{proof}

\begin{lemma}[A sufficient condition]
\label{lemma:schrodinger}
If $\ \Phi \in C^\infty$ and $\Phi(x) {\longrightarrow} +\infty$ when $|x| \rightarrow \infty$, the Schrödinger operator $-\Delta^L + \Phi $ on $\R^d$ has compact resolvent.
\end{lemma}
\begin{proof}
See \citep[Section 3]{Helffer2005witten} or \citep[Lemma XIII.67]{reed2012methods}.
\end{proof}
Now we can prove Proposition \ref{prop:minimizer}.
\begin{proof}[Proof of Proposition \ref{prop:minimizer}] 

We first prove that \eqref{eq:poinca_operator} admits a minimizer in $H^1(\mu)$. Indeed, we have, 
$$ \Poinca^{-1} = \inf_{f \in  (H^1 \cap L_0^2) \setminus \{0\}}\frac{\langle Lf,f \rangle_{L^2(\mu)}}{\|(\mathds{1}-\pi)f\|^2} = \inf_{f \in (H^1 \cap L_0^2) \setminus \{0\}}J(f),\ \textrm{ where }\, J(f):= \frac{\|\nabla f\|^2}{\|f\|^2}. $$
Let $(f_n)_{n\geqslant 0}$ be a sequence of functions in $H^1_0(\mu)$ equipped with the natural $H^1$-norm such that $(J(f_n))_{n\geqslant 0}$ converges to $\Poinca^{-1}$. As the problem in invariant by rescaling of $f$, we can assume that $\forall n \geqslant 0, \ \|f_n\|_{L^2(\mu)}^2=1$. Hence $J(f_n) = \|\nabla f_n\|_{L^2(\mu)}^2$ converges (to $\Poinca^{-1}$). In particular $\|\nabla f_n\|_{L^2(\mu)}^2$ is bounded in $L^2(\mu)$, hence $(f_n)_{n\geqslant 0}$ is bounded in $H^1(\mu)$. Since by Lemma \ref{lemma:compact_embedding} and \ref{lemma:schrodinger} we have a compact injection of $H^1(\mu)$ in $L^2(\mu)$, it holds, upon extracting a subsequence, that there exists $f \in H^1(\mu)$ such that 
\begin{align*}
\begin{cases}
f_n \hspace*{-0.15cm}&\rightarrow f \qquad \textrm{ strongly in }L^2(\mu) \\
 f_n \hspace*{-0.15cm}&\rightharpoonup  f \hspace*{0.65cm} \textrm{ weakly in }H^1(\mu).
\end{cases}
\end{align*}
Thanks to the strong ${L^2(\mu)}$ convergence, $\|f\|^2 = \underset{n \infty}{\textrm{lim }} \|f_n\|^2 = 1$. By the Cauchy-Schwarz inequality and then taking the limit  $n \rightarrow+\infty $, $$\|\nabla f\|^2=\underset{n \infty}{\textrm{lim }}\langle \nabla f_n,\nabla f \rangle \leqslant\underset{n \infty}{\textrm{lim }} \|\nabla f\|\|\nabla f_n\| = \|\nabla f\| \Poinca^{-1}. $$
Therefore, $ \|\nabla f\| \leqslant \Poinca^{-1/2}$ which implies that $J(f) \leqslant \Poinca^{-1}$, and so $J(f) = \Poinca^{-1}$. This shows that $f$ is a minimizer of $J$.

Let us next prove the PDE characterization of minimizers. A necessary condition on a minimizer $f_*$ of the problem $\inf_{f \in H^1(\mu)} \{\|\nabla f\|_{L^2(\mu)},\ \|f\|^2 = 1 \}$ is to satisfy the following Euler-Lagrange equation: there exists $\beta \in \R$ such that:
\begin{align*}
L f_* + \beta f_* &= 0.
\end{align*}
Plugging this into \eqref{eq:poinca_operator}, we have: $\Poinca^{-1} = \langle Lf_*,f_* \rangle =  -\beta \langle f_*,f_* \rangle = -\beta \|f_*\|^2_2= -\beta $. Finally, the equation satisfied by $f_*$ is:
\begin{align*}
Lf = -\Delta^L f_* + \langle \nabla V,\nabla f_* \rangle = \Poinca^{-1} f_*,
\end{align*}
which concludes the proof.
\end{proof}

\subsubsection{Case where \texorpdfstring{$d\mu$}{mu} has compact support}

We suppose in this section that $d\mu$ has a compact support included in $\Omega$. Without loss of generality we can take a set $\Omega$ with a $C^\infty$ smooth boundary $\partial \Omega$. In this case, without changing the result of the variational problem, we can restrict ourselves to functions that vanish at the boundary, namely the Sobolev space $H^1_D (\R^d,d\mu) = \left\{f \in H^1(\mu) \textrm{ s.t. } f_{|\partial \Omega} = 0  \right\}$. Note that, as $V$ is smooth, $H^1(\mu) \supset H^1 (\R^d,d\lambda)$ the usual "flat" space equipped with $d\lambda$, the Lebesgue measure. Note also that only in this section the domain of the operator $L$ is $H^2 \cap H^1_D$.

\begin{prop}[Properties of the minimizer in the compact support case]
\label{prop:minimizer_compact}
The problem \eqref{eq:poinca_operator} admits a minimizer in $H_D^1$ and every minimizer $f$ satisfies the partial differential equation:
\begin{align}
\label{eq:minimizer_compact}
L f = \Poinca^{-1} f.
\end{align} 
\end{prop}

\begin{proof}
The proof is exactly the same than the one of Proposition \ref{prop:minimizer} since $H^1_D$ can  be compactly injected in $L^2$ without any additional assumption on $V$.
\end{proof}

Let us take in this section $\h = H^d(\R^d,\,d\lambda)$, which is the RKHS associated to the kernel $k(x,x') = \mathrm{e}^{-\left\|x-x'\right\|}$. As $f_*$ satisfies \eqref{eq:minimizer_compact}, from regularity properties of elliptic PDEs, we infer that $f_*$ is $C^\infty(\overline{\Omega})$. By the Whitney extension theorem \citep{Whitney1932}, we can extend $f_*$ defined on $\overline{\Omega}$ to a smooth and compactly supported function in $\Omega' \supset \Omega$ of $\R^d$. Hence $f_* \in C^\infty_c(\R^d) \subset \h$. 
\begin{prop}
\label{prop:bias_compact}
Consider a minimizer $f_*$ of \eqref{eq:poinca_operator}. Then
\begin{align}
\Poinca^{-1} \leqslant \Poinca_{\lambda}^{-1} \leqslant \Poinca^{-1} + \lambda \frac{\|f_*\|_\h}{\ \ \|f_*\|^2_{L^2(\mu)}}.
\end{align}
\end{prop}

\begin{proof}
First note that $f_*$ has mean zero with respect to $d\mu$. Indeed, $\int f d\mu = \Poinca^{-1} \int L f d \mu = 0$, by the fact that $d\mu$ is the stationary distribution of the dynamics. 

For $\lambda > 0$,
\begin{align*}
\Poinca^{-1} \leqslant \Poinca_{\lambda}^{-1} &= \inf_{f \in \h \setminus \R\mathds{1}} \frac{\int_{\R^d}\|\nabla f(x)\|^2 d\mu(x) + \lambda \|f\|_\h^2}{\int_{\R^d}f(x)^2 d\mu(x) - \left(\int_{\R^d}f(x) d\mu(x)\right)^2 } \\
&\leqslant \frac{\int_{\R^d}\|\nabla f_*(x)\|^2 d\mu(x) + \lambda \|f_*\|_\h^2}{\int_{\R^d}f_*(x)^2 d\mu(x) }= \Poinca^{-1} + \lambda \frac{\|f_*\|_\h}{\ \ \|f_*\|^2_{L^2(\mu)}},
\end{align*}
which provides the result.
\end{proof}
\section{Technical inequalities}
\label{sec:technical_inequalities}

\subsection{Concentration inequalities}

We first begin by recalling some concentration inequalities for sums of random vectors and operators.

\begin{prop}[Bernstein’s inequality for sums of random vectors]
\label{prop:Bernstein_vector}
Let $z_1,\hdots, z_n$ be a sequence of independent identically and distributed random elements of a separable Hilbert space
$\h$. Assume that $\E \|z_1\|<+\infty$ and note $\mu = \E z_1$. Let $\sigma, L \geqslant 0$ such that, $$ \forall p \geqslant 2, \qquad\E \left\|z_1-\mu\right\| ^p_\h \leqslant \frac{1}{2} p!\sigma^2L^{p-2}.$$ 
Then, for any $\delta \in (0,1]$,
\begin{align}
\left\|\frac{1}{n}\sum_{i=1}^n z_i - \mu\right\|_\h\leqslant \frac{2L \log(2/\delta)}{n} + \sqrt{\frac{2 \sigma^2 \log(2/\delta)}{n}},
\end{align}
with probability at least $1-\delta$.
\end{prop}

\begin{proof}
This is a restatement of Theorem 3.3.4 of \citep{Yurinsky1995}.
\end{proof}

\begin{prop}[Bernstein’s inequality for sums of random operators]
\label{prop:Bernstein_operator}
Let $\h$ be a separable Hilbert space and let $X_1, \hdots, X_n$ be a sequence of independent and identically distributed
self-adjoint random operators on $\h$. Assume that  $\E(X_i) = 0$ and that there exist $T > 0$ and $S$ a
positive trace-class operator such that $\|X_i\| \leqslant T$ almost surely and $\E X_i^2 \preccurlyeq S$ for any $i \in \{1, \hdots, n\}$. Then, for any $\delta \in (0,1]$, the following inequality holds:
\begin{align}
\label{eq:concentration_operator}
\left\|\frac{1}{n}\sum_{i=1}^n X_i\right\| \leqslant \frac{2T\beta}{3n} + \sqrt{\frac{2 \|S\|\beta}{n}},
\end{align}
with probability at least $1-\delta$ and where $\beta = \log \frac{2 \mathrm{ Tr } S}{\|S\|\delta}$.
\end{prop}

\begin{proof}
The theorem is a restatement of Theorem 7.3.1 of \citep{Tropp2012UserFriendlyTF} generalized to the separable
Hilbert space case by means of the technique in Section 4 of \citep{Minsker2011}.
\end{proof}

\subsection{Operator bounds}
\label{subsec:operators}

\begin{lemma} 
Under assumptions \asm{asm:regularity_RKHS} and \asm{asm:bounded_kernel}, $\C$, $C$ and $\Delta$ are trace-class operators.
\end{lemma}
\begin{proof}
We only prove the result for $\Delta$, the proof for $\C$ and $C$ being similar. Consider an orthonormal basis $(\phi_i)_{i\in \mathbb{N}}$ of $\h$. Then, as $\Delta$ is a positive self adjoint operator,
\begin{align*}
\tr\ \Delta &= \sum_{i = 1}^\infty \langle \Delta \phi_i,\phi_i \rangle = \sum_{i = 1}^\infty \E_{\mu} \left[\sum_{j=1}^d\langle \partial_j K_x,\phi_i \rangle^2\right] = \E_{\mu} \left[\sum_{i = 1}^\infty \sum_{j=1}^d\langle \partial_j K_x,\phi_i \rangle^2\right] \\
&= \E_{\mu} \left[\sum_{j=1}^d\left\|\partial_j K_x\right\|^2\right] \leqslant \mathcal{K}_d.
\end{align*}
Hence, $\Delta$ is a trace-class operator.
\end{proof}

The following quantities are useful for the estimates in this section:
\begin{align*}
\Nla = \sup_{x \in \mathrm{supp} (\mu) } \left\| \D_{\lambda}^{-1/2} K_x \right\|^2_\h, \ \textrm{and} \quad \Fla = \sup_{x \in \mathrm{supp} (\mu) } \left\| \D_{\lambda}^{-1/2} \nabla K_x \right\|^2_\h.
\end{align*}
Note that under assumption \asm{asm:bounded_kernel}, $\Nla \leqslant \frac{\mathcal{K}}{\lambda}$ and $\Fla \leqslant \frac{\mathcal{K}_d}{\lambda}$. Note also that under refined assumptions on the spectrum of $\Delta$, we could have a better dependence of the latter bounds with respect to $\lambda$. Let us now state three useful lemmas to bound the norms of the operators that appear during the proof of Proposition \ref{prop:hat_P_to_P_lambda}.

\begin{lemma}
\label{lemma:concentration_C}
For any $\lambda > 0$ and any $\delta \in (0,1]$,
\begin{align*}
\left\|\D_\lambda^{-1/2}  (\widehat{C}-C) \D_\lambda^{-1/2}\right\|&\leqslant \frac{4 \Nla \log \frac{2\,\mathrm{Tr } \C }{\Poinca^\lambda_{\mu} \lambda \delta } }{3n} + \left[\frac{ 2\ \Poinca^\lambda_{\mu}\ \Nla \log \frac{2\,\mathrm{Tr } \C }{\Poinca^\lambda_{\mu} \lambda \delta }}{n}\right]^{1/2} \\ 
&\hspace*{0.5cm}+ 8 \Nla \left( \frac{\log(\frac{2}{\delta})}{n} + \sqrt{\frac{\log(\frac{2}{\delta})}{n}} \right) \\
&\hspace*{0.5cm}+ 16 \Nla \left( \frac{\log(\frac{2}{\delta})}{n} + \sqrt{\frac{\log(\frac{2}{\delta})}{n}} \right)^2,
\end{align*}
with probability at least $1-\delta$.
\end{lemma}

\begin{proof}[Proof of Lemma \ref{lemma:concentration_C}]
We apply some concentration inequality to the operator $\D_\lambda^{-1/2} \widehat{C}\D_\lambda^{-1/2}$ whose mean is exactly $\D_\lambda^{-1/2} C\D_\lambda^{-1/2}$. The calculation is the following:
\begin{align*}
\left\|\D_\lambda^{-1/2}  (\widehat{C}-C) \D_\lambda^{-1/2}\right\| &= \left\|\D_\lambda^{-1/2} \widehat{C} \D_\lambda^{-1/2} - \D_\lambda^{-1/2} C \D_\lambda^{-1/2} \right\| \\ 
&\leqslant \left\|\D_\lambda^{-1/2} \widehat{\C} \D_\lambda^{-1/2} - \D_\lambda^{-1/2} \C \D_\lambda^{-1/2} \right\| \\ 
& \hspace*{1cm}+ \left\|\D_\lambda^{-1/2} (\widehat{m}\otimes\widehat{m}) \D_\lambda^{-1/2} - \D_\lambda^{-1/2} (m\otimes m ) \D_\lambda^{-1/2} \right\| \\ 
&= \left\| \frac{1}{n}\sum_{i=1}^n \left[ (\D_\lambda^{-1/2} K_{x_i})\otimes (\D_\lambda^{-1/2} K_{x_i}) - \D_\lambda^{-1/2} \C \D_\lambda^{-1/2} \right] \right\| \\ 
& \hspace*{1cm}+ \left\| (\D_\lambda^{-1/2}\widehat{m})\otimes(\D_\lambda^{-1/2}\widehat{m})  -  (\D_\lambda^{-1/2}m)\otimes (\D_\lambda^{-1/2}m )  \right\|.
\end{align*}
We estimate the two terms separately.

\noindent {\bfseries Bound on the first term:} we use Proposition \ref{prop:Bernstein_operator}. To do this, we bound for $i \in \llbracket 1,n\rrbracket$ :
\begin{align*}
\left\|(\D_\lambda^{-1/2} K_{x_i})\otimes (\D_\lambda^{-1/2} K_{x_i}) - \D_\lambda^{-1/2} \C \D_\lambda^{-1/2}\right\| &\leqslant \left\| \D_{\lambda}^{-1/2} K_{x_i} \right\|^2_\h + \left\|\D_\lambda^{-1/2} \C \D_\lambda^{-1/2}\right\|\\
&\leqslant 2\Nla,
\end{align*}
and, for the second order moment,
\begin{align*}
\E &\left((\D_\lambda^{-1/2} K_{x_i})\otimes (\D_\lambda^{-1/2} K_{x_i}) - \D_\lambda^{-1/2} \C \D_\lambda^{-1/2}\right)^2 \\
&\hspace*{1cm}= \E \left[\left\| \D_{\lambda}^{-1/2} K_{x_i} \right\|^2_\h  (\D_\lambda^{-1/2} K_{x_i})\otimes (\D_\lambda^{-1/2} K_{x_i})\right] - \D_\lambda^{-1/2} \C \D_\lambda^{-1} \C \D_\lambda^{-1/2} \\
&\hspace*{1cm}\preccurlyeq \Nla \D_\lambda^{-1/2} \C \D_\lambda^{-1/2}.
\end{align*}
We conclude this first part of the proof by some estimation of the constant $\beta = \log \frac{2\,\mathrm{Tr } (\C \D_\lambda^{-1})}{\left\|\D_\lambda^{-1/2} \C \D_\lambda^{-1/2}\right\| \delta }$. Using $\mathrm{Tr } \C \D_\lambda^{-1} \leqslant \lambda^{-1} \mathrm{Tr } \C $, it holds $\beta \leqslant \log \frac{2\,\mathrm{Tr } \C }{\Poinca^\lambda_\mu \lambda \delta }$. Therefore,
\begin{align*}
&\left\| \frac{1}{n}\sum_{i=1}^n \left[ (\D_\lambda^{-1/2} K_{x_i})\otimes (\D_\lambda^{-1/2} K_{x_i}) - \D_\lambda^{-1/2} \C \D_\lambda^{-1/2} \right] \right\| \\
&\hspace*{3cm}\leqslant \frac{4 \Nla \log \frac{2\,\mathrm{Tr } \C }{\Poinca^\lambda_\mu \lambda \delta } }{3n} + \left[\frac{ 2\ {\Poinca^\lambda_\mu}\ \Nla \log \frac{2\,\mathrm{Tr } \C }{\Poinca^\lambda_\mu \lambda \delta }}{n}\right]^{1/2}.
\end{align*}
\noindent {\bfseries Bound on the second term.} Denote by $v = \D_\lambda^{-1/2} m$ and $\widehat{v} = \D_\lambda^{-1/2} \widehat{m}$. A simple calculation leads to
\begin{align*}
\|\widehat{v} \otimes \widehat{v} - v \otimes v\| &\leqslant \|v \otimes (\widehat{v} - v)\| + \| (\widehat{v} - v)\otimes v\| + \| (\widehat{v} - v)\otimes (\widehat{v} - v)\| \\
 &\leqslant 2 \|v\| \|\widehat{v} - v\| + \|\widehat{v} - v\|^2.
\end{align*}
We bound $\|\widehat{v} - v\|$ with Proposition \ref{prop:Bernstein_vector}. It holds: $\widehat{v} - v = \Delta_{\lambda}^{-1/2} (\widehat{m}-m) = \frac{1}{n}\sum_{i=1}^n\Delta_{\lambda}^{-1/2} (K_{x_i}-m) = \frac{1}{n}\sum_{i=1}^nZ_i$, with $Z_i = \Delta_{\lambda}^{-1/2} (K_{x_i}-m)$. Obviously for any $i \in \llbracket1, n \rrbracket$, $\E (Z_i) = 0$, and $\|Z_i\| \leqslant \|\Delta_\lambda^{-1/2} K_{x_i}\| + \|\Delta_\lambda^{-1/2} m\| \leqslant 2 \sqrt{\Nla}$. Furthermore, 
\begin{align*}
\E \|Z_i\|^2 = \E \left\langle \Delta_{\lambda}^{-1/2} (K_{x_i}-m), \Delta_{\lambda}^{-1/2} (K_{x_i}-m) \right\rangle &= \E \left\|\Delta_{\lambda}^{-1/2} K_{x_i}\right\|^2 - \left\|\Delta_{\lambda}^{-1/2} m\right\|^2 \\ 
&\leqslant \Nla.
\end{align*} 
Thus, for $p \geqslant 2$,
$$ \E \|Z_i\|^p \leqslant  \E \left(\|Z_i\|^{p-2} \|Z_i\|^2\right) \leqslant \frac{1}{2} p! \left( \sqrt{\Nla} \right)^{2} \left(2 \sqrt{\Nla} \right)^{p-2},  $$
hence, by applying Proposition \ref{prop:Bernstein_vector} with $L = 2\sqrt{\Nla}$ and $\sigma = \sqrt{\Nla} $,
\begin{align*}
\|\widehat{v} - v\| &\leqslant \frac{4 \sqrt{\Nla}\log (2/\delta)}{n} + \sqrt{\frac{2\Nla \log(2/\delta)}{n}} \\ 
&\leqslant  4 \sqrt{\Nla} \left( \frac{\log(2/\delta)}{n} + \sqrt{\frac{\log(2/\delta)}{n}} \right).
\end{align*}
Finally, as $\|v\| \leqslant \sqrt{\Nla}$,
\begin{align*}
\|\widehat{v} \otimes \widehat{v} - v \otimes v\| &\leqslant 8 \Nla \left( \frac{\log(2/\delta)}{n} + \sqrt{\frac{\log(2/\delta)}{n}} \right)\\ 
&\hspace*{1cm}+ 16 \Nla \left( \frac{\log(2/\delta)}{n} + \sqrt{\frac{\log(2/\delta)}{n}} \right)^2. 
\end{align*}
This concludes the proof of Lemma~\ref{lemma:concentration_C}.
\end{proof}

\begin{lemma}
\label{lemma:concentration_delta}
For any $\lambda \in \left(0,\|\Delta\|\,\right] $ and any $\delta \in (0,1]$,
\begin{align*}
\left\| \D_\lambda^{-1/2}  (\widehat{\D}-\D) \D_\lambda^{-1/2} \right\| \leqslant \frac{4 \Fla \log \frac{4\,\mathrm{Tr } \D }{\lambda \delta }}{3n} + \sqrt{\frac{ 2\ \Fla \log \frac{4\,\mathrm{Tr } \D }{\lambda \delta }}{n}},
\end{align*}
with probability at least $1-\delta$.
\end{lemma}

\begin{proof}[Proof of Lemma \ref{lemma:concentration_delta}]
 
As in the proof of Lemma \ref{lemma:concentration_C}, we want to apply some concentration inequality to the operator $\D_\lambda^{-1/2} \widehat{\Delta}\D_\lambda^{-1/2}$, whose mean is exactly $\D_\lambda^{-1/2} \Delta\D_\lambda^{-1/2}$. The proof is almost the same as Lemma~\ref{lemma:concentration_C}. We start by writing
\begin{align*}
\left\|\D_\lambda^{-1/2}  (\widehat{\D}-\D) \D_\lambda^{-1/2}\right\| &= \left\|\D_\lambda^{-1/2} \widehat{\D} \D_\lambda^{-1/2} - \D_\lambda^{-1/2} \D \D_\lambda^{-1/2} \right\| \\ 
&= \left\| \frac{1}{n}\sum_{i=1}^n \left[ (\D_\lambda^{-1/2} \nabla K_{x_i})\otimes (\D_\lambda^{-1/2}\nabla K_{x_i}) - \D_\lambda^{-1/2} \Delta \D_\lambda^{-1/2} \right] \right\|. 
\end{align*}
In order to use Proposition \ref{prop:Bernstein_operator}, we bound for $i \in \llbracket 1,n\rrbracket$,
\begin{align*}
\left\|(\D_\lambda^{-1/2} \nabla K_{x_i})\otimes (\D_\lambda^{-1/2} \nabla K_{x_i}) - \D_\lambda^{-1/2} \D \D_\lambda^{-1/2}\right\| &\leqslant \left\| \D_{\lambda}^{-1/2}\nabla K_{x_i} \right\|^2_\h + \left\|\D_\lambda^{-1/2} \D \D_\lambda^{-1/2}\right\| \\
 &\leqslant 2\Fla,
\end{align*}
and, for the second order moment,
\begin{align*}
&\E \left[\left((\D_\lambda^{-1/2} \nabla K_{x_i})\otimes (\D_\lambda^{-1/2} \nabla K_{x_i}) - \D_\lambda^{-1/2} \D \D_\lambda^{-1/2}\right)^2\right] \\
&= \E \left[\left\| \D_{\lambda}^{-1/2} \nabla K_{x_i} \right\|^2_\h  (\D_\lambda^{-1/2} \nabla K_{x_i})\otimes (\D_\lambda^{-1/2} \nabla K_{x_i})\right] - \D_\lambda^{-1/2} \D \D_\lambda^{-1} \D \D_\lambda^{-1/2} \\
& \preccurlyeq \Fla \D_\lambda^{-1/2} \D \D_\lambda^{-1/2}.
\end{align*}
We conclude by some estimation of $\beta = \log \frac{2\,\mathrm{Tr } (\D \D_\lambda^{-1})}{\left\|\D_\lambda^{-1} \D \right\| \delta }$. Since $\mathrm{Tr } (\D \D_\lambda^{-1}) \leqslant \lambda^{-1} \mathrm{Tr } \D $ and for $\lambda \leqslant \|\D\|$, $\left\|\D_\lambda^{-1} \D \right\| \geqslant 1/2$, it follow that $\beta \leqslant \log \frac{4\,\mathrm{Tr } \D }{\lambda \delta }$. The conclusion then follows from \eqref{eq:concentration_operator}.
\end{proof}

\begin{lemma}[Bounding operators]
\label{lemma:magic_2}
For any $\lambda > 0$, $\delta \in (0,1)$, and $n \geqslant 15 \Fla \log \frac{4\,\mathrm{Tr } \D }{\lambda \delta }$,
\begin{align*}
\left\|\widehat{\D}_\lambda^{-1/2} \D_\lambda^{1/2} \right\|^2 \leqslant 2,
\end{align*}
with probability at least $1-\delta$.
\end{lemma}

The proof of this result relies on the following lemma (see proof in \citep[Proposition 8]{rudi2017generalization}).

\begin{lemma}
\label{lemma:auxi_lemma}
Let $\h$ be a separable Hilbert space, $A$ and $B$ two bounded self-adjoint positive linear operators on $\h$ and $\lambda > 0$. Then
\begin{align*}
\left\|(A+\lambda I )^{-1/2} (B+\lambda I)^{1/2}\right\|\leqslant (1-\beta)^{-1/2},
\end{align*}
with $\beta = \lambda_{\rm{max}}\left((B+\lambda I)^{-1/2} (B-A) (B+\lambda I)^{-1/2}\right) < 1$, where $\lambda_{\rm{max}}(O)$ is the largest eigenvalue of the self-adjoint operator $O$.
\end{lemma}
We can now write the proof of Lemma~\ref{lemma:magic_2}.
\begin{proof}[Proof of Lemma \ref{lemma:magic_2}]
Thanks to Lemma \ref{lemma:auxi_lemma}, we see that $$\left\|\widehat{\D}_\lambda^{-1/2} \D_\lambda^{1/2} \right\|^2 \leqslant \left(1-\lambda_{\mathrm{max}}\left(\D_\lambda^{-1/2} (\widehat{\Delta} - \Delta) \D_\lambda^{-1/2}\right)\right)^{-1},$$ and as $ \left\|\D_\lambda^{-1/2} (\widehat{\Delta} - \Delta) \D_\lambda^{-1/2}\right\| < 1 $, we have: $$\left\|\widehat{\D}_\lambda^{-1/2} \D_\lambda^{1/2} \right\|^2 \leqslant \left(1-\left\|\D_\lambda^{-1/2} (\widehat{\Delta} - \Delta) \D_\lambda^{-1/2}\right\|\right)^{-1}.$$ We can then apply the bound of Lemma \ref{lemma:concentration_delta} to obtain that, if $\lambda$ is such that $\frac{4 \Fla \log \frac{4\,\mathrm{Tr } \D }{\lambda \delta }}{3n} + \sqrt{\frac{ 2\ \Fla \log \frac{4\,\mathrm{Tr } \D }{\lambda \delta }}{n}} \leqslant \frac{1}{2}$, then $\left\|\widehat{\D}_\lambda^{-1/2} \D_\lambda^{1/2} \right\|^2 \leqslant 2$ with probability $1-\delta$. The condition on $\lambda$ is satisfied when $n \geqslant 15 \Fla \log \frac{4\,\mathrm{Tr } \D }{\lambda \delta }$.

\end{proof}

\section{Calculation of the bias in the Gaussian case}
\label{sec:gaussian_bias}
We can derive a rate of convergence when $\mu$ is a one-dimensional Gaussian. Hence, we consider the one-dimensional distribution $d\mu$ as the normal distribution with mean zero and variance $1/(4a)$. Let $b>0$, we consider also the following approximation $\displaystyle \mathcal{P}_\kappa^{-1} = \inf_{f \in \h} \frac{\E_{\mu} (f'^2) + \kappa \| f\|_\h^2 }{\var_{\mu}(f)}$ where $\h$ is the RKHS associated with the Gaussian kernel $\exp(-b(x-y)^2)$. Our goal is to study how $\mathcal{P}_\kappa$ tends to $\mathcal{P}$ when $\kappa$ tends to zero. 

\begin{prop}[Rate of convergence for the bias in the one-dimensional Gaussian case]
\label{prop:gaussian_case}
If $d\mu$ is a one-dimensional Gaussian of mean zero and variance $1/(4a)$ there exists $A>0$ such that, if $\lambda \leqslant A$, it holds
\begin{align}
\Poinca^{-1} \leqslant \Poinca_\lambda^{-1} \leqslant \Poinca^{-1}(1+B \lambda \ln^2(1/\lambda)),
\end{align}
where $A$ and $B$ depend only on the constant $a$.
\end{prop}
We will show it by considering a specific orthonormal basis of $L^2(\mu)$, where all operators may be expressed simply in closed form.

\subsection{An orthonormal basis of \texorpdfstring{$L^2(\mu)$}{blo} and \texorpdfstring{$\h$}{PDFstring}}

We begin by giving an explicit a basis of $L^2(\mu)$ which is also a basis of $\h$.

\begin{prop}[Explicit basis]
We consider
\begin{align*}
f_{i}(x) =  \Big(\frac{c}{a}\Big)^{1/4} \big( 2^{i} i! \big)^{-1/2} \mathrm{e}^{-(c-a)x  ^2} H_{i}\left(\sqrt{2c} x\right),
\end{align*}
where $H_i$ is the $i$-th Hermite polynomial, and $c = \sqrt{ a^2 + 2ab}$. Then,
\begin{itemize}
\item $(f_i)_{i \geqslant 0}$ is an orthonormal basis of $L^2(\mu)$;
\item $\tilde{f}_i = \lambda_i^{1/2} f_i$ forms an orthonormal basis of $\h$, with $\lambda_i = \sqrt{ \frac{2a}{a+b+c} } \Big( \frac{b}{a+b+c} \Big) ^i$.
\end{itemize}
\end{prop}

\begin{proof}

We can check that this is indeed an orthonormal basis of $L^2(\mu)$:
\begin{align*}
\langle f_k, f_m \rangle_{L^2(\mu)} & = \int_\R \frac{1}{\sqrt{2 \pi / 4 a}} \mathrm{e}^{-2a x^2} \Big(\frac{c}{a}  \Big)^{1/2} \mathrm{e}^{-2 (c-a) x ^2} \big( 2^{k} k!\big)^{-1/2} \big( 2^{m} m! \big)^{-1/2} H_{k}(\sqrt{2c} x) H_{m}(\sqrt{2c} x) dx \\
& = \sqrt{2 c / \pi} \big( 2^{k} k!   \big)^{-1/2} \big( 2^{m} m!  \big)^{-1/2} \int_\R  \mathrm{e}^{-2c x^2} H_{k}(\sqrt{2c} x) H_{m}(\sqrt{2c} x) dx \\
& = \delta_{mk},
\end{align*}
using properties of Hermite polynomials. Considering the integral operator $T:  L^2(\mu) \to L^2(\mu)$, defined as $Tf(y) = \int_\R \mathrm{e}^{-b(x-y)^2} f(x) d\mu(x)$, we have:
\begin{align*}
T f_k(y)  & = \Big(\frac{c}{a} \Big)^{1/4}\big(2^{k} k!  \big)^{-1/2}\int_\R \mathrm{e}^{-(c-a)x ^2} H_{k}(\sqrt{2c} x)\frac{1}{\sqrt{2 \pi / 4 a}} \mathrm{e}^{-2a x^2} \mathrm{e}^{ - b(x-y)^2}  dx \\
& = \Big(\frac{c}{a} \Big)^{1/4}\big(2^{k} k!  \big)^{-1/2}   \mathrm{e}^{ - by^2} \frac{1}{\sqrt{2 \pi / 4 a}}\frac{1}{\sqrt{2c} }\int_\R \mathrm{e}^{-(a+b+c) x  ^2} H_{k}(\sqrt{2c} x) \mathrm{e}^{ 2b xy}  \sqrt{2c} dx \\
& = \Big(\frac{c}{a} \Big)^{1/4}\big(2^{k} k!  \big)^{-1/2}   \mathrm{e}^{ - by^2} \frac{1}{\sqrt{2 \pi / 4 a}}\frac{1}{\sqrt{2c} }\int_\R \mathrm{e}^{-\frac{a+b+c}{2c} x  ^2} H_{k}(x) \mathrm{e}^{ \frac{2b}{\sqrt{2c}} xy } dx.
\end{align*}
We consider $u$ such that
$\frac{1}{1-u^2} = \frac{a+b+c}{2c}$, that is, $ 1 - \frac{2c}{a+b+c} = \frac{a+b-c}{a+b+c}
= \frac{b^2}{(a+b+c)^2} = u^2$, which implies that $u = \frac{b}{a+b+c}$; 
and then $\frac{2u}{1-u^2} = \frac{b}{c}$.

Thus, using properties of Hermite polynomials (see Section~\ref{sec:hermite_facts}), we get:
\begin{align*}
T f_k(y)   & =  \Big(\frac{c}{a} \Big)^{1/4}\big(2^{k} k!  \big)^{-1/2}   \mathrm{e}^{ - by^2} \frac{1}{\sqrt{2 \pi / 4 a}}\frac{1}{\sqrt{2c} }  {\sqrt{\pi}}  {\sqrt{1-u^2}}   { H_k(\sqrt{2c} y)   }  \exp\left( \frac{u^2}{1-u^2} 2cy^2 \right) u^k \\ 
& =  \Big(\frac{c}{a} \Big)^{1/4}\big(2^{k} k!  \big)^{-1/2}   \frac{1}{\sqrt{2 \pi / 4 a}}\frac{1}{\sqrt{2c} }  {\sqrt{\pi}}  \frac{\sqrt{2c}}{\sqrt{a+b+c}}   { H_k(\sqrt{2c} y)   }  \exp(bu  y^2 - by^2) u^k \\
& =  \Big(\frac{c}{a} \Big)^{1/4}\big(2^{k} k!  \big)^{-1/2}     \frac{\sqrt{2a}}{\sqrt{a+b+c}}   { H_k(\sqrt{2c}y)   }  \exp\Big( - b   y^2   + 2c y^2\left( - 1 + \frac{1}{1-u^2} \right) \Big)u^k \\
& =  \frac{\sqrt{2a}}{\sqrt{a+b+c}}   \Big( \frac{b}{a+b+c} \Big) ^k f_k(y) \\
& =  \lambda_k f_k(y).
\end{align*}
This implies that $(\tilde{f_i})$ is an orthonormal basis of $\h$.
\end{proof}

We can now rewrite our problem in this basis, which is the purpose of the following lemma:

\begin{lemma}[Reformulation of the problem in the basis]
Let $(\alpha_i)_i \in \ell^2(\N)$. For $f = \sum_{i=0}^\infty
 \alpha_i f_i$, we have:
\begin{itemize}
\item $\displaystyle\| f\|_\h^2 = \sum_{i=0}^\infty \alpha_i^2 \lambda_i^{-1} = \alpha^\top \Diag(\lambda)^{-1} \alpha$;
\item $\displaystyle \textstyle\var_{\mu}(f(x)) = \displaystyle \sum_{i=0}^\infty \alpha_i^2 - \big(\displaystyle\sum_{i=0}^\infty \eta_i \alpha_i \big)^2 =  \alpha^\top ( I - \eta \eta^\top) \alpha$;
\item $\displaystyle \E_{\mu} f'(x)^2 = \sum_{i=0}^\infty \sum_{j=0}^\infty  \alpha_i \alpha_j (M^\top M)_{ij} = \alpha^\top M^\top M \alpha$,
\end{itemize}
where $\eta$ is the vector of coefficients of $\ \mathbf{1}_{_{L^2(\mu)}}$ and $M$ the matrix of coordinates of the derivative operator in the $(f_i)$ basis. The problem can be rewritten under the following form:
\begin{align}
\Poinca_\kappa^{-1} = \inf_\alpha \frac{ \alpha^\top ( M^\top M +\kappa \Diag(\lambda)^{-1} ) \alpha   } {\alpha^\top ( I - \eta \eta^\top ) \alpha },
\end{align}
where
\begin{itemize}
\item $\forall k \geqslant 0, \eta_{2k}\displaystyle=\left(\frac{c}{a} \right)^{1/4} \sqrt{ \frac{2a}{a+c}} \left( \frac{b}{a+b+c} \right)^{k}  \frac{\sqrt{(2k)!}}{ 2^{k} k!}$ and $\eta_{2k+1} = 0$ 
\item $\displaystyle\forall i \in \N, \left(M^\top M\right)_{ii} = \frac{1}{c}\left(2i(a^2+c^2) + (a-c)^2\right)$ and $\left(M^\top M\right)_{i,i+2} = \displaystyle\frac{1}{c}\left((a^2-c^2)\sqrt{(i+1)(i+2)}\right)$.
\end{itemize}
\end{lemma}

\begin{proof}
\noindent \textbf{Covariance operator.} Since $(f_i)$ is orthonormal for $L^2(\mu)$, we only  need to compute for each $i$, $\eta_i = \E_{\mu} f_i(x)$, as follows (and using properties of Hermite polynomials):
\begin{align*}
\eta_i = \langle 1, f_i \rangle_{L^2(\mu)}& =\Big(\frac{c}{a} \Big)^{1/4}\big(2^{i} i!  \big)^{-1/2} \int_\R \mathrm{e}^{-(c-a)x  ^2} H_{i}(\sqrt{2c} x)  \mathrm{e}^{-2ax^2} \sqrt{ 2a /\pi} dx  \\
& =\Big(\frac{c}{a} \Big)^{1/4}\big(2^{i} i!  \big)^{-1/2}  \sqrt{ a /(\pi c)} \int_\R \mathrm{e}^{-\frac{a+c}{2c} x  ^2} H_{i}( x) dx  \\
&  =\Big(\frac{c}{a} \Big)^{1/4}\big(2^{i} i!  \big)^{-1/2}\sqrt{ \frac{2a}{a+c}} \Big( \frac{c-a}{c+a} \Big)^{i/2} H_i(0) \mathrm{i}^i .
\end{align*}
This is only non-zero for $i$ even, and
\begin{align*}
\eta_{2k} & = \Big(\frac{c}{a} \Big)^{1/4}\big(2^{2k} (2k) !  \big)^{-1/2}\sqrt{ \frac{2a}{a+c}} \Big( \frac{c-a}{c+a} \Big)^{k} H_{2k}(0) (-1)^k \\
& = \Big(\frac{c}{a} \Big)^{1/4}\big(2^{2k} (2k) !  \big)^{-1/2}\sqrt{ \frac{2a}{a+c}} \Big( \frac{c-a}{c+a} \Big)^{k}  \frac{(2k)!}{ k!} \\
& = \Big(\frac{c}{a} \Big)^{1/4} \sqrt{ \frac{2a}{a+c}} \Big( \frac{c-a}{c+a} \Big)^{k}  \frac{\sqrt{(2k)!}}{ 2^{k} k!} \\
& =\Big(\frac{c}{a} \Big)^{1/4} \sqrt{ \frac{2a}{a+c}} \Big( \frac{b}{a+b+c} \Big)^{k}  \frac{\sqrt{(2k)!}}{ 2^{k} k!}.
\end{align*}
Note that  we must have $\sum_{i=0}^\infty \eta_i^2 = \| 1 \|_{L^2(\mu)}^2 = 1,$ which can indeed be checked ---the shrewd reader will recognize the entire series development of $(1-z^2)^{-1/2}$.

\noindent \textbf{Derivatives.} We have, using the recurrence properties of Hermite polynomials:
\begin{align*}
  f_i' & =\frac{a-c}{\sqrt{c}} \sqrt{i+1} f_{i+1} + \frac{a+c}{\sqrt{c}} \sqrt{i} f_{i-1},
\end{align*}
for $i>0$, while for $i=0$, $  f_0' = \frac{a-c}{\sqrt{c}}   f_1$. Thus, if $M$ is the matrix of coordinates of the derivative operator in the basis $(f_i)$, we have $M_{i+1,i} =  \frac{a-c}{\sqrt{c}} \sqrt{i+1}$ and $M_{i-1,i} =  \frac{a+c}{\sqrt{c}} \sqrt{i}$. This leads to
\begin{align*}
\langle f_i', f_j' \rangle_{L^2(\mu)} & = (M^\top M)_{ij}.
\end{align*}
We have
\begin{align*}
(M^\top M)_{ii} & = \langle f_i', f_i' \rangle_{L^2(\mu)} \\
& = \frac{1}{c} \Big((i+1) ( a-c)^2 + i ( a+c)^2\Big) \\
& = \frac{1}{c} \Big(2 i ( a^2+c^2) + (a-c)^2\Big) \mbox{ for } i\geqslant 0, \\
(M^\top M)_{i,i+2} & = \langle f_i', f_{i+2}' \rangle_{L^2(\mu)} \\
& = \frac{1}{c} \Big( (a^2 - c^2) \sqrt{(i+1)(i+2)} \Big) \mbox{ for } i\geqslant 0. 
\end{align*}
Note that we have $M \eta = 0$ as these are the coordinates of the derivative of the constant function (this can be checked directly by computing $(M\eta)_{2k+1} = M_{2k+1,2k}\eta_{2k} + M_{2k+1,2k+2} \eta_{2k+2}$).

\end{proof}

\subsection{Unregularized solution}

Recall that we want to solve $\displaystyle \mathcal{P}^{-1} = \inf_f \frac{\E_{\mu} f'(x)^2 }{\var_{\mu}(f(x))},$. The following lemma characterizes the optimal solution completely.
\begin{lemma}[Optimal solution for one dimensional Gaussian]
We know that the solution of the Poincaré problem is $\mathcal{P}^{-1} = 4a$ which is attained for $f_\ast(x) = x$. The decomposition of $f_\ast$ is the basis $(f_i)_i$ is given by $\displaystyle f_\ast = \sum_{i \geqslant 0} \nu_i f_i$, where $\forall k \geqslant 0$, $\nu_{2k} = 0$ and $\nu_{2k+1} = \Big(\frac{c}{a} \Big)^{1/4}\frac{ \sqrt{  a } }{  2c }  \big( \frac{2c}{a+c}\big)^{3/2}  \big( \frac{b}{a+b+c} \big)^{k} \frac{\sqrt{ (2k+1)!}}{2^{k} k!}.$
\end{lemma}
\begin{proof}
We thus need to compute:
\begin{align*}
\nu_i & = \langle f_\ast , f_i \rangle_{L^2(\mu)} \\
& =\Big(\frac{c}{a} \Big)^{1/4}\big(2^{i} i!  \big)^{-1/2} \int_\R \mathrm{e}^{-(c-a)x  ^2} H_{i}(\sqrt{2c} x)  \mathrm{e}^{-2ax^2} \sqrt{ 2a /\pi}  x dx \\
& =\Big(\frac{c}{a} \Big)^{1/4}\big(2^{i} i!  \big)^{-1/2} \sqrt{ 2a /\pi} \int_\R \mathrm{e}^{-(c+a)x  ^2} H_{i}(\sqrt{2c} x)     x dx  \\
& =\Big(\frac{c}{a} \Big)^{1/4}\big(2^{i} i!  \big)^{-1/2} \sqrt{ 2a /\pi}   \frac{1}{ {2c}} \int_\R \mathrm{e}^{-\frac{c+a}{2c} x  ^2} H_{i}(  x) x dx  \\
& =\Big(\frac{c}{a} \Big)^{1/4}\big(2^{i} i!  \big)^{-1/2} \sqrt{ 2a /\pi}   \frac{1}{  {4c}} \int_\R \mathrm{e}^{-\frac{c+a}{2c} x  ^2} [ H_{i+1}(  x)     + 2i H_{i-1}(x) ] dx  \\
& =\Big(\frac{c}{a} \Big)^{1/4}\big(2^{i} i!  \big)^{-1/2} \sqrt{ 2a /\pi}   \frac{\sqrt{\pi}}{4c} \sqrt{ \frac{2c}{a+c}} \Big(\big( \frac{c-a}{c+a} \big)^{(i+1)/2} H_{i+1}(0)  \mathrm{i}^{i+1} \\
&\hspace*{4cm}+2i \big( \frac{c-a}{c+a} \big)^{(i-1)/2} H_{i-1}(0)  \mathrm{i}^{i-1} \Big) ,
\end{align*}
which is only non-zero for  $i$ odd. We have:
\begin{align*}
\nu_{2k+1} & = \Big(\frac{c}{a} \Big)^{1/4}\big(2^{2k+1} (2k+1)!  \big)^{-1/2} \sqrt{ 2a /\pi}   \frac{\sqrt{\pi}}{4c}\sqrt{ \frac{2c}{a+c}} \Big(\big( \frac{c-a}{c+a} \big)^{k+1} H_{2k+2}(0)  (-1)^{k+1} \\
&\hspace*{6cm}+2(2k+1) \big( \frac{c-a}{c+a} \big)^{k} H_{2k}(0)  (-1)^k\Big) \\
& = \Big(\frac{c}{a} \Big)^{1/4}\big(2^{2k+1} (2k+1)!  \big)^{-1/2} \sqrt{ 2a /\pi}   \frac{\sqrt{\pi}}{4c} \sqrt{ \frac{2c}{a+c}} \Big(\big( \frac{c-a}{c+a} \big)^{k+1} H_{2k+2}(0)  (-1)^{k+1} \\
&\hspace*{6cm}+2(2k+1) \big( \frac{c-a}{c+a} \big)^{k} H_{2k}(0)  (-1)^k\Big) \\
& = \Big(\frac{c}{a} \Big)^{1/4}\big(2^{2k+1} (2k+1)!  \big)^{-1/2} \sqrt{ 2a /\pi}   \frac{\sqrt{\pi}}{4c} \sqrt{ \frac{2c}{a+c}}  \big( \frac{c-a}{c+a} \big)^{k} (-1)^k \\
&\hspace*{6cm}\Big(\big( \frac{c-a}{c+a} \big) 2(2k+1) H_{2k}(0)   +2(2k+1)     H_{2k}(0)   \Big) \\
& = \Big(\frac{c}{a} \Big)^{1/4}\big(2^{2k+1} (2k+1)!  \big)^{-1/2} \sqrt{ 2a /\pi}   \frac{\sqrt{\pi}}{4c} \sqrt{ \frac{2c}{a+c}}  \big( \frac{c-a}{c+a} \big)^{k} (-1)^k 2(2k+1)  H_{2k}(0)  \frac{2c}{c+a}\\
& = \Big(\frac{c}{a} \Big)^{1/4}\big(2^{2k+1} (2k+1)!  \big)^{-1/2} \sqrt{  a }   \frac{1}{  c \sqrt{ 2}} \big( \frac{2c}{a+c}\big)^{3/2}  \big( \frac{c-a}{c+a} \big)^{k} (-1)^k (2k+1)  H_{2k}(0)  \\
& = \Big(\frac{c}{a} \Big)^{1/4}\big(2^{2k+1} (2k+1)!  \big)^{-1/2} \sqrt{  a }   \frac{1}{  c \sqrt{ 2}}  \big( \frac{2c}{a+c}\big)^{3/2}  \big( \frac{c-a}{c+a} \big)^{k} (2k+1)  \frac{ (2k)!}{k!}\\
& = \Big(\frac{c}{a} \Big)^{1/4}\frac{ \sqrt{  a } }{  2c } \big( \frac{2c}{a+c}\big)^{3/2}  \big( \frac{c-a}{c+a} \big)^{k} \frac{\sqrt{ (2k+1)!}}{2^{k} k!}\\
& = \Big(\frac{c}{a} \Big)^{1/4}\frac{ \sqrt{  a } }{  2c }  \big( \frac{2c}{a+c}\big)^{3/2}  \big( \frac{b}{a+b+c} \big)^{k} \frac{\sqrt{ (2k+1)!}}{2^{k} k!}.
\end{align*}
\end{proof} 
Note that we have:
\begin{align*}
\mu^\top \nu & = \langle 1 , f_\ast \rangle_{L^2(\mu)} = 0 \\
\| \nu \|^2 & = \| f_\ast \|_{L^2(\mu)}^2 = \frac{1}{4a} \\
M^\top M \nu & = 4a \nu.
\end{align*}
The first equality if obvious from the odd/even sparsity patterns. The third one can be checked directly. The second one can probably be checked by another shrewd entire series development.
  
If we had $\nu^\top \Diag(\lambda)^{-1} \nu $ finite, then we would have
\begin{align*}
\mathcal{P}^{-1} \leqslant \mathcal{P}^{-1}_\kappa \leqslant \mathcal{P}^{-1} \left( 1 + \kappa \cdot \nu^\top \Diag(\lambda)^{-1} \nu\right) ,
\end{align*}
which would very nice and simple. Unfortunately, this is not true (see below).

\subsubsection{Some further properties for \texorpdfstring{$\nu$}{PDFstring}}

We have:
$\frac{c-a}{c+a} = \frac{b}{a+b+c}
$, and the following equivalent  $\frac{\sqrt{ \sqrt{k}(2k/e)^{2k+1}}}{2^{k}\sqrt{k} (k/e)^k} \sim \frac{k^{1/4+k+1/2}}{k^{k+1/2}} \sim k^{1/4}$ (up to constants). Thus 
\begin{align*}
|\nu_{2k+1}^2 \lambda_{2k+1}^{-1}| & \leqslant\Big(\frac{c}{a} \Big)^{1/2}\frac{  {  a } }{  c^2 } \big( \frac{2c}{a+c}\big)^{3}  \big( \frac{b}{a+b+c} \big)^{2k-2k-1}  \sqrt{ \frac{a+b+c}{2a} } \sqrt{k}= \Theta(\sqrt{k})
\end{align*}
hence,
\begin{align*}
\sum_{k=0}^{2m+1}
\nu_{k}^2 \lambda_{k}^{-1} \sim \Theta( m^{3/2}).
\end{align*}
Consequently, $\nu^\top \Diag(\lambda)^{-1} \nu = +\infty$.

Note that we have the extra recursion
\begin{align*}
\nu_k = \frac{1}{\sqrt{4c}} \big[\sqrt{k+1} \eta_{k+1} + \sqrt{k} \eta_{k-1}\big].
\end{align*}

\subsection{Truncation}
We are going to consider a truncated version $\alpha$, of $\nu$, with only the first $2m+1$ elements. That is $\alpha_k = \nu_k$ for $k\leqslant 2m+1$ and $0$ otherwise.

\begin{lemma}[Convergence of the truncation]
\label{lem:truncation}
Consider $g^m = \sum_{k = 0}^{\infty} \alpha_k f_k = \sum_{k = 0}^{2m+1} \nu_k f_k$, recall that $u = \frac{b}{a+b+c}$. For $m \geqslant \textrm{max}\{- \frac{3}{4 \ln u}, \frac{1}{6c}\}$, we have the following:
\begin{enumerate}[label=(\roman*)]
\item \hspace*{0.5cm}  $\left|\|\alpha\|^2 - \frac{1}{4a}\right| \leqslant L m u^{2m}$
\item \hspace*{0.5cm}  $\alpha^\top \eta = 0$ 
\item \hspace*{0.5cm}  $\left|\alpha^\top M^\top M \alpha - 1\right| \leqslant L m^2 u^{2m}$
\item \hspace*{0.5cm}  $\alpha^\top \Diag(\lambda)^{-1} \alpha \leqslant L m^{3/2}$,
\end{enumerate}
where $L$ depends only on $a,b,c$.
\end{lemma}

\begin{proof} We show successively the four estimations.

\noindent \textbf{(i)} Let us calculate $\|\alpha\|^2$. We have: $\|\alpha\|^2 - \frac{1}{4a} = \|\alpha\|^2 - \|\nu\|^2 = \sum_{k = m+1}^\infty \nu^2_{2k+1}$. Recall that $u = \frac{b}{a+b+c} \leqslant 1$, by noting $A = \left(\frac{c}{a} \right)^{1/4}\frac{ \sqrt{  a } }{  2c }  \big( \frac{2c}{a+c}\big)^{3/2}$, we have 
$$\|\alpha\|^2 - \frac{1}{4a} = A^2 \sum_{k = m+1}^\infty \frac{(2k+1)!}{(2^k k!)^2} u^{2k}.$$
Now by Stirling inequality: 
\begin{align*}
\frac{(2k+1)!}{(2^k k!)^2} u^{2k} & \leqslant \frac{e \left(2k+1\right)^{2k+1+1/2}\mathrm{e}^{-(2k+1)}}{(\sqrt{2\pi} 2^k k^{k+1/2} \mathrm{e}^{-k})^2}u^{2k} \\
& = \frac{\sqrt{2}}{\pi} \left(1+ \frac{1}{2k}\right)^{2k+1}\left(k + \frac{1}{2}\right)^{1/2}u^{2k}. \\
& \leqslant \frac{4e}{\pi} \sqrt{k}u^{2k}. \\
\end{align*} 
And for $m \geqslant - \frac{1}{4 \ln u}$,
\begin{align*}
\sum_{m+1}^\infty \sqrt{k}u^{2k} &\leqslant \int_m^\infty \sqrt{x}u^{2x}dx \\
&\leqslant \int_m^\infty xu^{2x}dx \\
&= u^{2m} \frac{\( 1-2 m \ln u\)}{\( 2 \ln u\)^2}\\
&\leqslant \frac{m u^{2m}}{\ln(1/u)}.
\end{align*}
Hence finally:
$$\left|\|\alpha\|^2 - \frac{1}{4a}\right| \leqslant \frac{4A^2e}{\pi\ln (1/u)} m u^{2m}.$$

\noindent \textbf{(ii)} is straightforward because of the odd/even sparsity of $\nu$ and $\eta$.

\noindent \textbf{(iii)} Let us calculate $\|M\alpha\|^2$. We have: 
\begin{align*}
\|M\alpha\|^2 - 1 &= \|M\alpha\|^2 - \|M\nu\|^2 \\
&= \sum_{k,j \geqslant m+1} \nu_{2k+1}\nu_{2j+1} \(M^\top M\)_{2k+1,2j+1} \\
&= \sum_{k = m+1}^\infty \nu_{2k+1}^2 \(M^\top M\)_{2k+1,2k+1} + 2 \sum_{k = m+1}^\infty \nu_{2k+1}\nu_{2k+3} \(M^\top M\)_{2k+1,2k+3} \\
&= \frac{A^2}{c} \sum_{k = m+1}^\infty \frac{(2k+1)!}{(2^k k!)^2} \(2 (2k+1) ( a^2+c^2) + (a-c)^2\)u^{2k} \\
&\hspace*{2.5cm}-\frac{2A^2ab}{c} \sum_{k = m+1}^\infty \frac{\sqrt{(2k+1)!}}{(2^k k!)}\frac{\sqrt{(2k+3)!}}{(2^{k+1} (k+1)!)} \sqrt{(2k+2)(2k+3)}u^{2k+1}.
\end{align*}
Let us call the two terms $u_m$ and $v_m$ respectively. For the first term, when $m \geqslant \textrm{max}\{- \frac{3}{4 \ln u}, \frac{1}{6c}\}$ a calculation as in \textbf{(i)} leads to:
\begin{align*}
\left|u_m\right| &\leqslant \frac{24A^2e(u^2+c^2)}{\pi c}\int_{m}^\infty x\sqrt{x} u^{2x}dx + \frac{(a-c)^2}{c} \left( \|\alpha\|^2-\|\nu\|^2 \right) \\
&\leqslant \frac{24A^2e(u^2+c^2)}{\pi c}\int_{m}^\infty x^2 u^{2x}dx -\frac{4A^2e}{\pi\ln u} m u^{2m}\\
&= -\frac{24A^2e(u^2+c^2)}{\pi c}\frac{u^{2m} (2m\ln u  ( 2 m \ln(u)-2)+2)}{8\ln^3(u)} -\frac{4A^2e}{\pi\ln u} m u^{2m} \\
&\leqslant -\frac{12A^2e(a^2+c^2)}{\pi c \ln(u)}m^2u^{2m} -\frac{4A^2e}{\pi\ln u} m u^{2m} \\
&\leqslant -\frac{4A^2e}{\pi\ln u}  \(\frac{3(a^2+c^2)}{c}m+1\)m u^{2m} \\
&\leqslant \frac{24 A^2 c e }{\pi\ln (1/u)} m^2 u^{2m}.
\end{align*}
and for the second term, applying another time Stirling inequality, we get:

\begin{align*}
\frac{\sqrt{(2k+1)!}}{2^k k!}\frac{\sqrt{(2k+3)!}}{2^{k+1} (k+1)!} u^{2k+1} &\leqslant \frac{\mathrm{e}^{1/2} \left(2k+1\right)^{k+3/4}\mathrm{e}^{-(k+1/2)}}{\sqrt{2\pi} 2^k k^{k+1/2} \mathrm{e}^{-k}} \frac{\mathrm{e}^{1/2} \left(2k+3\right)^{k+7/4}\mathrm{e}^{-(k+3/2)}}{\sqrt{2\pi} 2^{k+1} (k+1)^{k+3/2} \mathrm{e}^{-(k+1)}}  u^{2k+1} \\
&\leqslant \frac{\left(2k+1\right)^{k+3/4}}{\sqrt{2\pi} 2^k k^{k+1/2} } \frac{ \left(2k+3\right)^{k+7/4}}{\sqrt{2\pi} 2^{k+1} (k+1)^{k+3/2} }  u^{2k+1} \\
&= \frac{\sqrt{2}}{\pi} \frac{\(1+\frac{1}{2k}\)^{k+3/4} \(1+\frac{3}{2k}\)^{k+7/4}}{\(1+\frac{1}{k}\)^{k+3/2}} \sqrt{k} u^{2k+1}  \\
&\leqslant \frac{\sqrt{2}}{\pi} \frac{\(1+\frac{3}{2k}\)^{2k} \(1+\frac{3}{2k}\)^{5/2}}{\(1+\frac{1}{k}\)^{k}\(1+\frac{1}{k}\)^{3/2}} \sqrt{k} u^{2k+1}  \\
&\leqslant \frac{\sqrt{2}}{\pi} \(1+\frac{3}{2k}\)^{2k} \(1+\frac{3}{2k}\)^{5/2} \sqrt{k} u^{2k+1}  \\
&\leqslant \frac{15e^3}{\pi} \sqrt{k} u^{2k+1}.  \\
\end{align*}
Hence, as $\displaystyle \sum_{k \geqslant m+1} \sqrt{k}u^{2k+1} \leqslant - \frac{mu^{2m+1}}{\ln u}$, we have
$\displaystyle |v_m| \leqslant \frac{30 A^2 ab e^3}{\pi c \ln (1/u)} m u^{2m}.$

\noindent \textbf{(iv)} Let us calculate $\alpha^\top \Diag(\lambda)^{-1} \alpha $. We have: 

\begin{align*}
\alpha^\top \Diag(\lambda)^{-1} \alpha &= \sum_{k = 0}^{m} \nu^2_{2k+1}\lambda^{-1}_{2k+1} \\
&= A^2\sqrt{\frac{bu}{2a}} \sum_{k = 0}^{m} \frac{(2k+1)!}{(2^k k!)^2} u^{2k}  u^{-(2k+1)}\\
&= A^2\sqrt{\frac{b}{2au}} \sum_{k = 0}^{m} \frac{(2k+1)!}{(2^k k!)^2} \\
&\leqslant \frac{4A^2e\sqrt{b}}{\pi \sqrt{2au}} \sum_{k = 0}^{m} \sqrt{k} \\
&\leqslant \frac{8A^2e\sqrt{b}}{\pi \sqrt{2au}} m^{3/2}. \\
\end{align*}
\noindent \textbf{(Final constant.)} By taking $L = \displaystyle \max \left\{ \frac{4A^2e}{\pi\ln (1/u)}, \frac{48 A^2 c e }{\pi\ln (1/u)}, \frac{60 A^2 ab e^3}{\pi c \ln (1/u)}, \frac{8A^2e\sqrt{b}}{\pi \sqrt{2au}}\right\}$, we have proven the lemma.
\end{proof}
We can now state the principal result of this section:
\begin{prop}[Rate of convergence for the bias]
If $\kappa \leqslant \min\{a^2, 1/5, u^{1/(3c)}\}$ and such that $\ln(1/\kappa)\kappa \leqslant \frac{\ln (1/u)}{2aL}$, then
\begin{align}
\Poinca^{-1} \leqslant \Poinca_\kappa^{-1} \leqslant \Poinca^{-1}\left(1+\frac{L}{2\ln^2(1/u)} \kappa \ln^2(1/\kappa)\right).
\end{align}
\end{prop}
\begin{proof}
The first inequality $\Poinca^{-1} \leqslant \Poinca_\kappa^{-1}$ is obvious. On the other side, 
\begin{align*}
\Poinca_\kappa^{-1} = \inf_\beta \frac{ \beta^\top ( M^\top M +\kappa \Diag(\lambda)^{-1} ) \beta   } {\beta^\top ( I - \eta \eta^\top ) \beta } \leqslant \frac{ \alpha^\top ( M^\top M +\kappa \Diag(\lambda)^{-1} ) \alpha   } {\alpha^\top ( I - \eta \eta^\top ) \alpha },
\end{align*}
With the estimates of Lemma~\ref{lem:truncation}, we have for $m u^{2m} <\frac{1}{4aL}$:
\begin{align*}
\Poinca_\kappa^{-1} &\leqslant \frac{ 1+ Lm^2u^{2m} + \kappa L m^{3/2}   } {\frac{1}{4a} - Lmu^{2m}} \\
&\leqslant \Poinca^{-1}(1+ Lm^2u^{2m} + \kappa L m^{3/2}).
\end{align*}
Let us take $m=\frac{\ln(1/\kappa)}{2 \ln (1/u)}$.Then
\begin{align*}
\Poinca_\kappa^{-1} &\leqslant \Poinca^{-1}(1+ \kappa L\frac{\ln^2(1/\kappa)}{4 \ln^2 (1/u)} + \kappa L \frac{\ln^{3/2}(1/\kappa)}{2^{3/2} \ln^{3/2} (1/u)}) \\
&\leqslant \Poinca^{-1}\left(1+ \kappa L\frac{\ln^2(1/\kappa)}{2 \ln^2 (1/u)}\right),
\end{align*}
as soon as $\kappa \leqslant a^2$. Note also that the condition $m u^{2m} <\frac{1}{4aL}$ can be rewritten in terms of $m$ as $\kappa \ln(1/\kappa) <\frac{\ln(1/u)}{2aL}$. The other conditions of Lemma \ref{lem:truncation} are $\kappa \leqslant \mathrm{e}^{-3/2} \sim 0.22 $ and $\kappa \leqslant u^{1/(3c)}$

\end{proof}
 
\subsection{Facts about Hermite polynomials}
\label{sec:hermite_facts}

\paragraph{Orthogonality.}
We have:
$$
\int_\R \mathrm{e}^{- x^2 } H_k(x) H_m(x) = 2^k k! \sqrt{\pi} \delta_{km}.
$$

\paragraph{Recurrence relations.}
We have:
$$
H_i'(x) = 2i H_{i-1}(x),
$$
and
$$
H_{i+1}(x) = 2x H_i(x) - 2i H_{i-1}(x).
$$
\paragraph{Mehler's formula.}
We have:
\begin{align*}
\sum_{k=0}^\infty \frac{ H_k (x) \mathrm{e}^{-x^2 / 2 } H_k(y) \mathrm{e}^{-y^2 /2 } }{  2^k k! \sqrt{\pi} } u^k
& = \frac{1}{\sqrt{\pi}} \frac{1}{\sqrt{1-u^2}}
\exp \Big(
\frac{2u}{1+u} xy - \frac{u^2}{1-u^2} ( x-y)^2 - \frac{x^2}{2}    - \frac{ y^2}{2} \Big).
\end{align*}

This implies that
the functions $x \mapsto  \frac{1}{\sqrt{\pi}} \frac{1}{\sqrt{1-u^2}}
\exp \Big(
\frac{2u}{1+u} xy - \frac{u^2}{1-u^2} ( x-y)^2 - \frac{x^2}{2}    - \frac{ y^2}{2} 
\Big)$ has coefficients 
$\frac{ H_k(y) \mathrm{e}^{-y^2/2} }{  \sqrt{ 2^k k! \sqrt{\pi} } } u^k$
in the orthonormal basis $( x \mapsto \frac{ H_k(x) \mathrm{e}^{-x^2/2} }{  \sqrt{ 2^k k! \sqrt{\pi} } })$ of $L_2(dx)$.

Thus
$$
\int_\R \frac{1}{\sqrt{\pi}} \frac{1}{\sqrt{1-u^2}}
\exp \Big(
\frac{2u}{1+u} xy - \frac{u^2}{1-u^2} ( x-y)^2 - \frac{x^2}{2}    - \frac{ y^2}{2} 
\Big)  
\frac{ H_k(x) \mathrm{e}^{-x^2/2} }{  \sqrt{ 2^k k! \sqrt{\pi} } }dx = \frac{ H_k(y) \mathrm{e}^{-y^2/2} }{  \sqrt{ 2^k k! \sqrt{\pi} } } u^k,
$$
that is
 $$
\int_\R
\exp \Big(
\frac{2u}{1+u} xy - \frac{u^2}{1-u^2} ( x-y)^2 -  {x^2}    
\Big)  
 { H_k(x)  } dx = 
 {\sqrt{\pi}}  {\sqrt{1-u^2}}
   { H_k(y)   } u^k.
$$

This implies:

 $$
\int_\R
\exp \Big(
\frac{2u}{1-u^2} xy - \frac{x^2 }{1-u^2}     
\Big)  
 { H_k(x)  } dx = 
 {\sqrt{\pi}}  {\sqrt{1-u^2}}
   { H_k(y)   }  \exp( \frac{u^2}{1-u^2} y^2 ) u^k
$$

For $y=0$, we get
$$
\int_\R 
\exp \Big(
   -\frac{ x^2  }{1-u^2}
\Big)   H_k(x)    dx = 
 {\sqrt{\pi}}  {\sqrt{1-u^2}}
 H_k(0)   u^k .
 $$

Another consequence is that
\begin{align*}
\sum_{k=0}^\infty \frac{ H_k (x)  H_k(y)  }{  2^k k! \sqrt{\pi} } u^k
& = \frac{1}{\sqrt{\pi}} \frac{1}{\sqrt{1-u^2}}
\exp \Big(
\frac{2u(1-u) + 2u^2}{1-u^2} xy - \frac{u^2}{1-u^2} (x^2+y^2)  \Big)\\
& = \frac{1}{\sqrt{\pi}} \frac{1}{\sqrt{1-u^2}}
\exp \Big(
\frac{2u }{1-u^2} xy - \frac{u}{1-u^2} (x^2+y^2)  +  \frac{u}{1+u} (x^2+y^2) \Big)\\
& = \frac{1}{\sqrt{\pi}} \frac{1}{\sqrt{1-u^2}}
\exp \Big(
-\frac{u }{1-u^2}  (x-y)^2 \Big) \exp \Big( \frac{u}{1+u} (x^2+y^2) \Big) \\
& = \frac{1}{\sqrt{\pi}} \frac{\sqrt{u} }{\sqrt{1-u^2}}
\exp \Big(
-\frac{u }{1-u^2}  (x-y)^2 \Big) \frac{1}{\sqrt{u}} \exp \Big( \frac{u}{1+u} (x^2+y^2) \Big).
\end{align*}
Thus, when $u$ tends to 1, as a function of $x$, this tends to a Dirac at $y$ times $\mathrm{e}^{y^2}$.

\end{document}